\documentclass{amsart}%
\usepackage{amsfonts}
\usepackage{amssymb}
\usepackage{amsmath}
\usepackage{csquotes}
\usepackage{graphicx}%
\providecommand{\U}[1]{\protect\rule{.1in}{.1in}}
\newtheorem{theorem}{Theorem}[section]
\newtheorem{lemma}[theorem]{Lemma}
\newtheorem{definition}[theorem]{Definition}
\newtheorem{corollary}[theorem]{Corollary}
\newtheorem{open problem}[theorem]{Open Problem}
\newtheorem{example}[theorem]{Example}

\newtheorem{proposition}[theorem]{Proposition}
\newtheorem{remark}[theorem]{Remark}
\numberwithin{equation}{section}

\begin{document}
\title{Rosette Harmonic Mappings}
\author{Jane McDougall}
\address{Department of Mathematics and Computer Science, Colorado College, Colorado
Springs, Colorado 80903}
\email{jmcdougall@coloradocollege.edu}
\author{Lauren Stierman}
\address{Department of Mathematics and Computer Science, Colorado College, Colorado
Springs, Colorado 80903}
\email{l\_stierman@coloradocollege.edu}
\thanks{The authors were supported in part by Department of Mathematics and Computer
Science summer grant 2017 \& 2019 from Colorado College.}
\date{January 15th 2020}

\begin{abstract}
A harmonic mapping is a univalent harmonic function of one complex variable.
We define a family of harmonic mappings on the unit disk whose images are
rotationally symmetric ``rosettes" with $n$ cusps or $n$ nodes, $n\geq3$.
These mappings are analogous to the n-cusped hypocycloid, but are modified by
Gauss hypergeometric factors, both in the analytic and co-analytic parts.
Relative rotations by an angle $\beta$ of the analytic and anti-analytic parts
lead to graphs that have cyclic, and in some cases dihedral symmetry of order
$n.$ While the graphs for different $\beta$ can be dissimilar, the cusps are
aligned along axes that are independent of $\beta$. For certain isolated
values of $\beta,$ the boundary function is continuous with arcs of constancy,
and has nodes of interior angle $\pi/2-\pi/n$ instead of cusps.

\end{abstract}
\maketitle

\section{Introduction}

We introduce the rosette harmonic mappings, analogous to the $n$-cusped
hypocycloid mappings. For each integer $n$ where $n\geq3,$ we obtain a family
of mappings, that in many instances have only cyclic rather than dihedral
symmetry. The mappings have $n$ cusps, or in some cases $n$ nodes rather than
cusps. It is interesting to consider one particular mapping for each $n$ in
which the boundary of the mapping is continuous, \ but with arcs of constancy.
Our main goal is to establish the univalence of the rosette harmonic mappings.
Additionally we describe the location and orientation of cusps and nodes. We
also define a fundamental set from which the full graph of a rosette mapping
can be reconstructed, which is useful for computational  efficiency. 

We begin by establishing some notation and standard terminology associated
with planar harmonic mappings. A \textbf{harmonic mapping} $f$ is a complex
valued univalent harmonic function defined on a region in the complex plane $%
\mathbb{C}
$. Harmonic mappings can be arrived at in a variety of ways, for example by
adding different harmonic functions together, or by using the Poisson integral
formula, and more recently by using the shear construction, first described in
\cite{ClunieSheilSmall}. Univalence is not guaranteed however, except for in
the latter approach. For any harmonic mapping $f,$ we write $f=h+\bar{g}$
where $h$ and $g$ are analytic, and call $h$ and $g$ the \textbf{analytic} and
\textbf{co-analytic} parts of $f$, respectively. The decomposition is unique
up to the constant terms of $h$ and $g,$ and $h+\bar{g}$ is known as the
\textbf{canonical decomposition }of $f$\textbf{. \ }Our mappings are defined
on $U,$ the \textbf{open unit disk} in the complex plane. The Jacobian $J_{f}$
of $f$ is given by $J_{f}\left(  z\right)  =\left\vert f_{z}\left(  z\right)
\right\vert ^{2}-\left\vert f_{\bar{z}}\left(  z\right)  \right\vert
^{2}=\left\vert h^{\prime}\left(  z\right)  \right\vert ^{2}-\left\vert
g^{\prime}\left(  z\right)  \right\vert ^{2}$. We say that $f$ is
\textbf{sense-preserving} if $J_{f}\left(  z\right)  >0$ in $U.$ A theorem of
L\v{e}wy \cite{Lewy} states that a harmonic function $f$ is locally one-to-one
if $J_{f}$ is non-vanishing in $U$. Thus $f$ is locally one-to-one and
sense-preserving if and only if $\left\vert g\right\vert <\left\vert
h\right\vert $ and thus, there exists a meromorphic function known as the
\textbf{analytic dilatation} of $f,$ given by $\omega_{f}\left(  z\right)
=g^{\prime}/h^{\prime}.$ Note that the analytic dilatation is related to the
\textbf{complex dilatation }$\mu_{f}=\bar{g}^{\prime}/h^{\prime}$ from the
theory of quasiconformal mappings. We refer here to $\omega_{f}\left(
z\right)  $ simply as the \textbf{dilatation} of $f$ - for more information
see \cite{PeterHMBook} and \cite{ECA}. For a given complex valued function
$f$, we sometimes use the notation $\overline{f(}z)$ to denote $\overline
{f(z)}$, the complex conjugate of the number $f(z)$.

The rosette harmonic mappings of Definition \ref{defnharmonic} are
modifications of a simple harmonic mapping known as the hypocycloid harmonic
mapping, with image under the unit disk bounded by a $n$-cusped hypocycloid.
The symbol $\partial A$ here denotes the topological boundary of the set $A$.

\begin{example}
\label{hypo}Let $n\in%
\mathbb{N}
,$ $n\geq3.$ The \textbf{hypocycloid harmonic mapping} is defined on $U$ by
$f_{hyp}\left(  z\right)  =z+\frac{1}{n-1}\bar{z}^{n-1}$ with analytic and
co-analytic parts $z$ and $\frac{1}{n-1}z^{n-1}$. The dilatation is
$\omega_{f}\left(  z\right)  =z^{n-2}$ and $J_{f}\left(  z\right)
=1^{2}-\left\vert \bar{z}^{n-2}\right\vert ^{2}=1-\left\vert z^{2n-4}%
\right\vert .$ Clearly $J_{f}\left(  z\right)  >0$ in $U$ so $f$ is locally
one to one$.$ It is also univalent on $U$ (see for instance \cite{PeterHMBook}
or \cite{ECA}). Upon extension to $\bar{U},$ we can consider the boundary
curve $f\left(  e^{it}\right)  ,$ which for $n=4\ $is the familiar astroid
curve from calculus. We consider the boundary extension $f_{hyp}\left(
e^{it}\right)  =e^{it}+\frac{1}{n-1}e^{-i\left(  n-1\right)  t}$ which has
singular points (where the derivative is $0$ ) precisely when $t=2k\pi/n$ ,
$k=1,2,...,n$. The left of Figure \ref{rosette0} shows the image of $\bar{U}$
under $f_{hyp} $ for $n=6$, where $f_{hyp}$ maps the unit circle $\partial U$
onto a $6$-cusped hypocycloid.
\end{example}

%

\begin{figure}[h]%
\centering
\includegraphics[
height=2.3514in,
width=5.028in
]%
{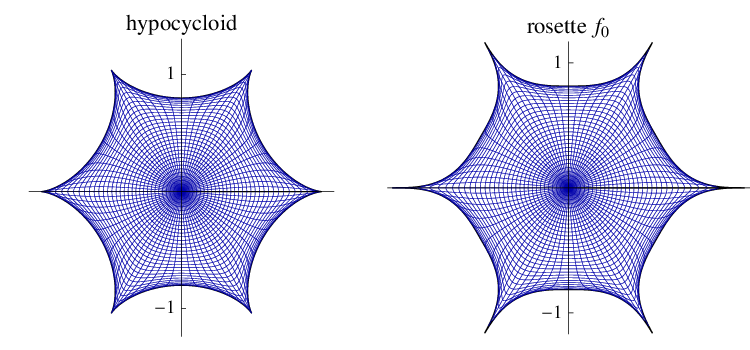}%
\caption{Images of a regular polar grid in $U$ under the 6-cusped hypocycloid
(left) and a 6-cusped rosette mapping (right) defined by $f_{0}\left(
z\right)  =z\,_{2}F_{1}\left(  \frac{1}{2},\frac{1}{12};\frac{13}{12}%
;z^{12}\right)  +\frac{1}{5}\bar{z}^{5}\,\overline{_{2}F_{1}\left(  \frac
{1}{2},\frac{5}{12};\frac{17}{12};z^{12}\right)  }$}%
\label{rosette0}%
\end{figure}

The rosette harmonic mappings introduced here can be viewed as modifications
of the hypocycloid mappings, and are formulated by incorporating Gauss
hypergeometric $_{2}F_{1}\,\ $factors into the analytic and co-analytic parts.
The rosette mappings $f_{\beta}$ will be defined in Section 3, but an example
of a 6-cusped rosette mapping appears on the right of Figure \ref{rosette0}.
In comparison with the 6-cusped hypocycloid, the rosette has cusps that are
more ``pointy". Figure \ref{five} indicates further examples of rosette
mappings for $n=6$ in which the images of the unit disk may have rotational
but not reflectional symmetry.

The process by which we obtain rosette harmonic mappings with essentially
different features, is by rotating the analytic and co-analytic parts relative
to one another. In one interesting configuration, the analytic and co-analytic
derivatives in the boundary extension alternate between ``alignment" and
``cancellation" on sub-arcs of $\partial U,$ leading to arcs of constancy on
the boundary of the unit disk. Other harmonic mappings with arcs of constancy
on the disk boundary are the Poisson extensions of piecewise constant
functions defined on the unit circle; these mappings have been studied in
\cite{Sheil-Small}, \cite{JaneBethPeter}, \cite{Mapping4}, and
\cite{MappingFull}, for instance.

In the forthcoming article \cite{AbdullahMcDougall}, we explore the rosette
minimal surfaces that ``lift" from the rosette harmonic mappings. The
hypocyloid mappings of Example \ref{hypo} can lift to an Enneper surface.
However the minimal graphs lifting from rosette harmonic mappings with the
arcs of constancy described above, have interesting similarities and contrasts
when compared with the Jenkins-Serrin surfaces. Jenkins-Serrin minimal
surfaces arise as ``lifts" of Poisson extensions of piecewise constant
boundary functions - the same mappings referenced in the previous paragraph.
The prototype for a Jenkins-Serrin surface is Schwarz's first surface (see
\cite{Schwarz}) which lifts from a harmonic mapping onto a square; more
general Jenkins-Serrin surfaces are described in \cite{JS}, and assume values
of either $+\infty$ or $-\infty$ over the boundary segments connecting the
vertices of the underlying harmonic mappings. Further Jenkins-Serrin surfaces
have been explored in \cite{DurenThygerson}, and \cite{MSStars}, for example.
The rosette minimal graphs turn out to have in common with Jenkins-Serrin
surfaces that the height function has a constant magnitude over the
curves/straightedges connecting the vertices of the underlying harmonic
mapping, and the completion of the minimal graphs contain vertical lines
(these lines occur where the height function changes between a positive and
negative value). These rosette minimal surfaces differ from the Jenkins-Serrin
surfaces in that the constant values of the height function are finite.

It would be interesting to know further examples of harmonic mappings on the
unit disk for which the boundary extensions have arcs of constancy mapping
onto vertices, for which the boundary correspondence is continuous, and for
which there exists a minimal surface lift with a piecewise constant height
function over the boundary of the harmonic mapping.

The sections following this introduction are organized as follows: In Section
2, the properties of the $_{2}F_{1}$ hypergeometric functions utilized in the
definition of the rosette harmonic mappings are described. In Section 3, the
rosette harmonic mappings are defined and their rotational and reflectional
symmetries are explored. In Section 4 we describe the boundary, that is
piecewise smooth between the nodes and cusps that together form the image of
the unit circle. In particular, we highlight the remarkable situation in which
the rosette harmonic mapping is constant on alternating arcs that partition
the unit circle. In Section 5 we use the argument principle for harmonic
functions to show that the rosette harmonic mappings are in fact univalent. We
also describe a computationally efficient way to construct the image of the
unit disk by using rotations of a smaller ``fundamental set".

\section{Hypergeometric functions\label{hyper}}

We begin by defining two Gauss hypergeometric $_{2}F_{1}$ functions.

\begin{definition}
\label{hyperDef}Let $U$ denote the unit disk. For $n\geq2$ and $z\in\bar{U},$
consider the Gauss $_{2}F_{1}$ hypergeometric functions
\begin{align*}
H_{n}\left(  z\right)   &  =\,_{2}F_{1}\left(  \frac{1}{2},\frac{1}%
{2n},1+\frac{1}{2n},z\right) \\
G_{n}\left(  z\right)   &  =\,_{2}F_{1}\left(  \frac{1}{2},\frac{1}{2}%
-\frac{1}{2n},\frac{3}{2}-\frac{1}{2n},z\right)
\end{align*}
for $n\geq2.$ Note when $n=2$, that $G_{2}=H_{2}$.
\end{definition}

By the Corollary in \cite{MerkesScott}, $H_{n}$ and $G_{n}$ map the unit disk
onto a convex region. Since $H_{n}\left(  0\right)  =G_{n}\left(  0\right)
=1,$ the convex regions $H_{n}\left(  U\right)  $ and $G_{n}\left(  U\right)
$ contain $1$. Thus $H_{n}$ and $G_{n}$ can be considered to be perturbations
of the constant function of unit value$.$ In Figure \ref{pert} where $n=6$,
one can see the distortion from unity in $H_{6}\left(  z\right)  $ and
$G_{6}\left(  z\right)  .$ The reflectional symmetry in the real axis is also
apparent. These and further properties of the hypergeometric functions $H_{n}
$ and $G_{n}$ are stated in Proposition \ref{hyperProps}.%

\begin{figure}[h]%
\centering
\includegraphics[
height=1.3699in,
width=3.9297in
]%
{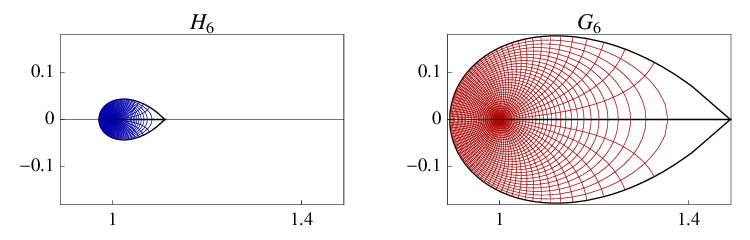}%
\caption{Images of a polar grid in $U$ under $H_{6}$ and $G_{6}.$}%
\label{pert}%
\end{figure}

\begin{proposition}
\label{hyperProps} Let $n\geq3.$ \newline(i) The Taylor coefficients $c_{m}$
of $H_{n}\left(  z\right)  \,$and $d_{m}$ of $G_{n}\left(  z\right)  $ are
given by the formulae
\[
c_{m}=A_{m}\frac{1}{2mn+1}\text{ and }d_{m}=A_{m}\frac{n-1}{n\left(
2m+1\right)  -1},\text{ }m\geq0,
\]
where $A_{0}=1,$ and $A_{m}=\left(
\begin{array}
[c]{c}%
2m-1\\
m-1
\end{array}
\right)  /2^{2m-1},$ for $m\geq1.$\newline(ii) The hypergeometric functions
$H_{n}$ and $G_{n}$ have reflectional symmetry
\[
H_{n}\left(  \bar{z}\right)  =\overline{H_{n}\left(  z\right)  }\text{ and
}G_{n}\left(  \bar{z}\right)  =\overline{G_{n}\left(  z\right)  }.
\]
(iii) Both $H_{n}\ $and $G_{n}$ are univalent mappings of the open unit disk
onto a bounded convex region in the right half plane. Moreover, both series
converge absolutely on the closed unit disk.\newline(iv) Define $K_{n}%
=H_{n}\left(  1\right)  $. Both $H_{n}\ $and $G_{n}$ are positive for
$x\in\left[  -1,1\right]  ,$ with the interval $\left[  H_{n}\left(
-1\right)  ,H_{n}\left(  1\right)  \right]  $ strictly contained in $\left[
G_{n}\left(  -1\right)  ,G_{n}\left(  1\right)  \right]  ,$ and the values at
$1$ are as follows:
\begin{align}
H_{n}\left(  1\right)   &  =K_{n}=\sqrt{\pi}\frac{\Gamma\left(  1+\frac{1}%
{2n}\right)  }{\Gamma\left(  \frac{1}{2}+\frac{1}{2n}\right)  },\label{Kn}\\
G_{n}\left(  1\right)   &  =\left(  n-1\right)  \tan\left(  \frac{\pi}%
{2n}\right)  K_{n}=\sqrt{\pi}\frac{\Gamma\left(  \frac{3}{2}-\frac{1}%
{2n}\right)  }{\Gamma\left(  1-\frac{1}{2n}\right)  }.\nonumber
\end{align}
Moreover,
\[
5/6<G_{n}\left(  -1\right)  <H_{n}\left(  -1\right)  <1<H_{n}\left(  1\right)
<G_{n}\left(  1\right)  <2.
\]

\end{proposition}

\begin{proof}
We first note the form of the Taylor coefficients $c_{m}$ of $H_{n},$ and
$d_{m}$ of $G_{n},$ where $n\geq2,$ and $m\geq0.$ By the definition of
hypergeometric series, and using the Pochhammer rising factorial symbol, we
obtain
\[
c_{m}=\frac{\left(  \frac{1}{2}\right)  _{m}\left(  \frac{1}{2n}\right)  _{m}%
}{m!\left(  1+\frac{1}{2n}\right)  _{m}}=A_{m}\frac{\left(  \frac{1}%
{2n}\right)  \left(  \frac{1}{2n}+1\right)  \left(  \frac{1}{2n}+2\right)
\ldots\left(  \frac{1}{2n}+m-1\right)  }{\left(  \frac{1}{2n}+1\right)
\left(  \frac{1}{2n}+2\right)  \ldots\left(  \frac{1}{2n}+m-1\right)  \left(
\frac{1}{2n}+m\right)  }%
\]
where we define $A_{m}=\left(  \frac{1}{2}\right)  _{m}/m!$ . All but two
factors cancel, leaving
\[
c_{m}=A_{m}\frac{\left(  \frac{1}{2n}\right)  }{\left(  \frac{1}{2n}+m\right)
}=A_{m}\frac{1}{2mn+1}.
\]
A more familiar formula for $A_{m}$ in terms of binomial coefficients is
obtained from
\[
A_{m}=\frac{1\cdot3\cdot5\ldots\left(  2m-1\right)  }{2^{m}m!}=\frac{\left(
2m-1\right)  !}{\left(  m-1\right)  !2^{m-1}2^{m}m!}=\left(
\begin{array}
[c]{c}%
2m-1\\
m-1
\end{array}
\right)  /2^{2m-1}.
\]
The mth Taylor coefficient of $G_{n}$ is
\[
d_{m}=\frac{\left(  \frac{1}{2}\right)  _{m}\left(  \frac{1}{2}-\frac{1}%
{2n}\right)  _{m}}{m!\left(  \frac{3}{2}-\frac{1}{2n}\right)  _{m}}=A_{m}%
\frac{\left(  \frac{1}{2}-\frac{1}{2n}\right)  _{m}}{\left(  \frac{3}{2}%
-\frac{1}{2n}\right)  _{m}}.
\]
Again, upon expanding the Pochhammer symbols and simplifying we obtain
\[
d_{m}=A_{m}\frac{\left(  \frac{1}{2}-\frac{1}{2n}\right)  }{\left(
\frac{2m+1}{2}-\frac{1}{2n}\right)  }=A_{m}\frac{n-1}{n\left(  2m+1\right)
-1},
\]
establishing (i). Clearly these are positive term series$,$ so the stated
symmetries involving conjugation in (ii) hold. We now prove (iv). Because the
coefficients $c_{m}$ and $d_{m}$ are positive for all $m=0,1,2,...,$ the real
line maps to the real line under both $H_{n}$ and $G_{n}.$ Moreover, for $x>0$
the function values are positive, and increasing with $x$. Clearly
$H_{n}\left(  1\right)  >1>H_{n}\left(  -1\right)  ,$ because the $c_{m}$ are
positive with $c_{0}=1,$ and for $x<0,$ we obtain an alternating series. For
$G_{n},$ we see that $d_{0}-d_{1}<G_{n}\left(  -1\right)  <d_{0}-d_{1}+d_{2}.$
Computing the lower bound, we obtain $d_{0}=1,$ and $d_{1}=A_{1}\frac{1}%
{2n+1}=\frac{1}{4n+2},$ so
\[
G_{n}\left(  -1\right)  >d_{0}-d_{1}=1-\frac{1}{2}\frac{n-1}{3n-1}>1-\frac
{1}{6}=5/6.
\]
To see that $G_{n}\left(  -1\right)  <H_{n}\left(  -1\right)  ,$ we again use
the alternating series test and show that the lower bound $c_{0}-c_{1}$ for
$H_{n}\left(  -1\right)  $ is larger than the upper bound $d_{0}-d_{1}+d_{2}$
for $G_{n}\left(  -1\right)  .$ We show that the equivalent inequality
$c_{1}+d_{2}\leq d_{1}$ holds$.$ Note that $A_{1}=\frac{1}{2}$ and
$A_{2}=\frac{3}{8},$ so the inequality becomes
\[
\frac{1}{2}\frac{1}{2n+1}+\frac{3}{8}\frac{n-1}{5n-1}\leq\frac{1}{2}\frac
{n-1}{3n-1}%
\]
which after some algebra is equivalent to $\left(  n-3\right)  \left(
22n^{2}-7n+1\right)  \geq0.$ The latter is true for $n\geq3,$ with equality at
$n=3.$ Thus the inequalities $G_{n}\left(  -1\right)  <H_{n}\left(  -1\right)
<1$ hold for $n\geq3$. Finally, coefficient by coefficient, the following
equivalent inequalities are equivalent to $c_{m}<d_{m}$ for $m=1,2,3,...$
\begin{align*}
A_{m}\frac{1}{2mn+1}  &  <A_{m}\frac{n-1}{n\left(  2m+1\right)  -1}\\
0  &  <2mn\left(  n-2\right)  ,
\end{align*}
and this last inequality is clearly true for all $m=1,2,...$ when $n>2$. Thus
$H_{n}\left(  1\right)  <G_{n}\left(  1\right)  .$ Finally we use Theorem 18
in \S 32$\ $of \cite{Rainville} to compute
\[
K_{n}=H_{n}\left(  1\right)  =\sqrt{\pi}\frac{\Gamma\left(  1+\frac{1}%
{2n}\right)  }{\Gamma\left(  \frac{1}{2}+\frac{1}{2n}\right)  }\text{ and
}G_{n}\left(  1\right)  =\sqrt{\pi}\frac{\Gamma\left(  \frac{3}{2}-\frac
{1}{2n}\right)  }{\Gamma\left(  1-\frac{1}{2n}\right)  }%
\]
Note that
\[
G_{n}\left(  1\right)  /H_{n}\left(  1\right)  =\frac{\Gamma\left(  \frac
{3}{2}-\frac{1}{2n}\right)  }{\Gamma\left(  1-\frac{1}{2n}\right)  }%
\frac{\Gamma\left(  \frac{1}{2}+\frac{1}{2n}\right)  }{\Gamma\left(
1+\frac{1}{2n}\right)  }=\frac{\Gamma\left(  \frac{1}{2}+\frac{1}{2n}\right)
\Gamma\left(  \frac{3}{2}-\frac{1}{2n}\right)  }{\Gamma\left(  1+\frac{1}%
{2n}\right)  \Gamma\left(  1-\frac{1}{2n}\right)  }.
\]
From the well known identities for the gamma function, $\Gamma\left(
z+1\right)  =z\Gamma\left(  z\right)  ,$ and $\Gamma\left(  z\right)
\Gamma\left(  1-z\right)  =\pi/\sin\left(  \pi z\right)  ,$ we obtain
\[
\Gamma\left(  1+z\right)  \Gamma\left(  1-z\right)  =\left(  \pi z\right)
/\sin\left(  \pi z\right)  .
\]
Therefore
\[
G_{n}\left(  1\right)  /H_{n}\left(  1\right)  =\frac{\pi\left(  \frac{1}%
{2}-\frac{1}{2n}\right)  }{\sin\left(  \pi\left(  \frac{1}{2}-\frac{1}%
{2n}\right)  \right)  }\frac{\sin\left(  \pi\frac{1}{2n}\right)  }{\pi\frac
{1}{2n}}=\left(  n-1\right)  \frac{\sin\left(  \pi/2n\right)  }{\cos\left(
\pi/2n\right)  }=\left(  n-1\right)  \tan\left(  \pi/2n\right)  .
\]

We finish by proving (iii). The radius of convergence of any Gauss
hypergeometric series is 1. However the convergence on the closed disk is
guaranteed for $_{2}F_{1}\left(  a,b,c,z\right)  $ whenever $\operatorname{Re}%
\left(  c-a-b\right)  >0$ (Section 29 of \cite{Rainville}); here
$\operatorname{Re}\left(  c-a-b\right)  =1/2\,\,$for both $G_{n}$ and $H_{n}$.
Because both $G_{n}$ and $H_{n}$ satisfy (ii) of the Corollary in
\cite{MerkesScott}, $_{2}F_{1}\left(  a,b,c,z\right)  $ maps the open disk
univalently onto a convex region. Because of positive Taylor coefficients, we
have $G_{n}\left(  -1\right)  <\operatorname{Re}G_{n}\left(  z\right)
<G_{n}\left(  1\right)  $ and $H_{n}\left(  -1\right)  <\operatorname{Re}%
H_{n}\left(  z\right)  <H_{n}\left(  1\right)  .$ Thus the images
$H_{n}\left(  U\right)  $ and $G_{n}\left(  U\right)  $ are convex sets within
the lines $\operatorname{Re}z=5/6$ and $\operatorname{Re}z=2.$
\end{proof}

Examples of the hypergeometric functions $H_{n}$ and $G_{n}$ are calculated in
\cite{AbdullahMcDougall}, through applying the shear construction of
\cite{ClunieSheilSmall} to a particular conformal mapping onto a regular
12-gon. This shear was in fact the origin of the first rosette mapping. The
functions $H_{n}$ and $G_{n}$ appear in antiderivatives associated with
several related\ radical expressions in \cite{AbdullahMcDougall}, and are
restated here.

\begin{proposition}
\label{hypergeometric} Let $n\geq2,\,$and let $z\in U.$ Then
\begin{align}
\int_{0}^{z}\frac{1}{\sqrt{1-\zeta^{2n}}}d\zeta &  =zH_{n}\left(
z^{2n}\right) \label{hypint1}\\
\int_{0}^{z}\frac{\zeta^{n-2}}{\sqrt{1-\zeta^{2n}}}d\zeta &  =\frac{z^{n-1}%
}{n-1}G_{n}\left(  z^{2n}\right)  . \label{hypint2}%
\end{align}

\end{proposition}

\begin{proof}
The integrations are carried out using the Gauss hypergeometric formula, as
indicated in \cite{AbdullahMcDougall}. The integrals are restated here in
terms of the notation introduced in Definition \ref{hyperDef}.
\end{proof}

\section{Rosette Harmonic Mappings}

The analytic antiderivatives of Proposition \ref{hypergeometric} are defined
below as $h_{n}\left(  z\right)  $ and $g_{n}\left(  z\right)  .$ These become
(rotations aside) the analytic and co-analytic parts of the rosette harmonic
mappings in Definition \ref{defnharmonic}. We note the similarity with the
analytic and co-analytic parts of the hypocycloid, modified here with
hypergeometric factors.

\begin{definition}
\label{defnhg}For $n\geq2$ and $n\in%
\mathbb{N}
$, define the analytic functions
\[
h_{n}\left(  z\right)  =z\,H_{n}\left(  z^{2n}\right)  ,\text{ \ }g_{n}\left(
z\right)  =\frac{z^{n-1}}{n-1}G_{n}\left(  z^{2n}\right)  ,\text{ }z\in\bar{U}%
\]
where $H_{n}$ and $G_{n}$ are the hypergeometric $_{2}F_{1}$ functions of
Definition \ref{hyperDef}.
\end{definition}

\begin{remark}
In the current context, we consider only $n\geq3$ for rosette harmonic
mappings, but we note that when $n=2,$ $h_{n}\left(  z\right)  =g_{n}\left(
z\right)  $ becomes the standard Schwarz-Christoffel mapping of the unit disk
onto a square.
\end{remark}

The function $h_{n}$ can easily be shown to be starlike and therefore
univalent on the unit disk. However $g_{n}$ is not univalent for $n\geq3$. The
image $h_{6}\left(  \bar{U}\right)  $ is shown on the left of Figure
\ref{canon6}, and a portion of the graph of $\overline{g_{6}}\left(  \bar
{U}\right)  $ is shown on the right, with $z$ restricted to $\left\{  z\in
\bar{U}:\arg z\in\left(  0,\left(  5/2\right)  \pi/n\right)  \right\}  $.%

\begin{figure}[h]%
\centering
\includegraphics[
height=2.3817in,
width=4.5074in
]%
{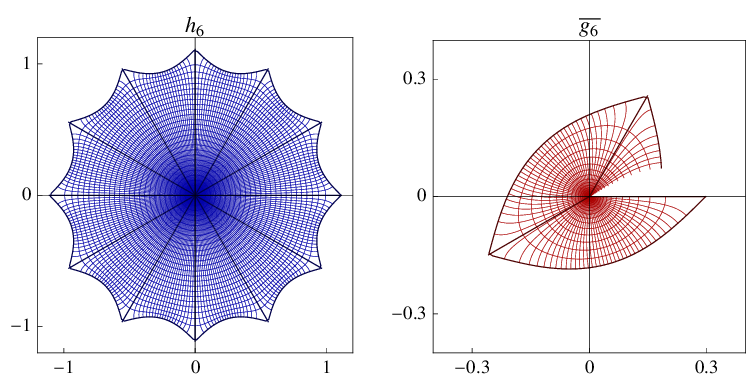}%
\caption{Image of $\bar{U}$ under $h_{6}$ (left)$,$ and a first-quadrant
sector of $\bar{U}$ under $\overline{g_{6}}$ (right)$.$ Note that the image of
$\overline{g_{6}}$ is relatively enlarged$.$ }%
\label{canon6}%
\end{figure}

\begin{proposition}
\label{analco}Let $n\geq3$ and let $\zeta$ be a primitive $2nth$ root of
unity$.$ Then $h_{n}\left(  z\right)  $ and $g_{n}\left(  z\right)  $ are
continuous on $\bar{U}$ and analytic on $\overline{U}\backslash\left\{
\zeta^{j}:j=1,2,...,2n\right\}  .$ Additionally the following hold.
\newline(i) Reflectional symmetries $h_{n}\left(  \bar{z}\right)
=\overline{h_{n}\left(  z\right)  }$ and $g_{n}\left(  \bar{z}\right)
=\overline{g_{n}\left(  z\right)  }$ are present.\newline(ii) The functions
$h_{n}\left(  e^{it}\right)  $ and $g_{n}\left(  e^{it}\right)  $ have
magnitudes $\left\vert H_{n}\left(  e^{i2nt}\right)  \right\vert $ and
\newline$\frac{1}{n-1}\left\vert G_{n}\left(  e^{i2nt}\right)  \right\vert ,$
and these magnitudes are each $\left(  \pi/n\right)  $-periodic as functions
of $t.$ Moreover, convergence at $H\left(  1\right)  $ and $G\left(  1\right)
$ is not uniform$.$\newline(iii) Both $h_{n}$ and $g_{n}$ exhibit $2n$-fold
rotational symmetry: for \ $j\in%
\mathbb{Z}
,$
\begin{equation}
h_{n}\left(  e^{ij\pi/n}z\right)  =e^{ij\pi/n}h_{n}\left(  z\right)  \text{
and }g_{n}\left(  e^{ij\pi/n}z\right)  =\left(  -1\right)  ^{j}e^{-ij\pi
/n}g_{n}\left(  z\right)  . \label{ga}%
\end{equation}
(iv) The derivatives of $h_{n}\left(  z\right)  $ and $g_{n}\left(  z\right)
$ are given by
\begin{equation}
h_{n}^{\prime}\left(  z\right)  =\frac{1}{\sqrt{1-z^{2n}}}\text{ and }%
g_{n}^{\prime}\left(  z\right)  =\frac{z^{n-2}}{\sqrt{1-z^{2n}}}.
\label{derivatives}%
\end{equation}

\end{proposition}

\begin{proof}
The stated symmetries in (i) follow from the reflections $H_{n}\left(  \bar
{z}^{2n}\right)  =\overline{H_{n}\left(  z^{2n}\right)  }$ and $G_{n}\left(
\bar{z}^{2n}\right)  =\overline{G_{n}\left(  z^{2n}\right)  }$ of Proposition
\ref{hyperProps} (ii). For (ii) we easily compute $\left\vert h_{n}\left(
e^{it}\right)  \right\vert =\left\vert e^{it}H_{n}\left(  e^{i2nt}\right)
\right\vert $ and $\left\vert g_{n}\left(  e^{it}\right)  \right\vert
=\left\vert e^{i\left(  n-1\right)  t}G_{n}\left(  e^{i2nt}\right)
\right\vert /\left(  n-1\right)  $ to see that $h_{n}\left(  e^{it}\right)  $
and $g_{n}\left(  e^{it}\right)  $ have the stated magnitudes. These
magnitudes have period $\pi/n$ as functions of $t,$ since each of
$H_{n}\left(  e^{it}\right)  $ and $G_{n}\left(  e^{it}\right)  $ have period
$2\pi$ as functions of $t.$ For (iii), first note that $\left(  e^{ij\pi
/n}z\right)  ^{2n}=z^{2n}$ and $\left(  e^{ij\pi/n}\right)  ^{n-1}=e^{ij\pi
}e^{-ij\pi/n}=\left(  -1\right)  ^{j}e^{-ij\pi/n}$ for any integer $j.$ Thus
\[
h_{n}\left(  e^{ij\pi/n}z\right)  =e^{ij\pi/n}zH_{n}\left(  \left(
e^{ij\pi/n}z\right)  ^{2n}\right)  =e^{ij\pi/n}h_{n}\left(  z\right)  ,\text{
and}%
\]%
\[
g_{n}\left(  e^{ij\pi/n}z\right)  =\frac{1}{n-1}\left(  e^{ij\pi/n}z\right)
^{n-1}G_{n}\left(  \left(  e^{ij\pi/n}z\right)  ^{2n}\right)  =e^{ij\left(
\pi-\pi/n\right)  }g_{n}\left(  z\right)  .
\]
\newline For (iv), the derivatives are immediate from (\ref{hypint1}) and
(\ref{hypint2}). The stated region of analyticity of the functions $h_{n}$ and
$g_{n}$ is due to the fact that their derivatives can be analytically
continued across the boundary, except at the 2nth roots of unity where
derivatives do not exist. Moreover the convergence of $H\left(  1\right)  $
and $G\left(  1\right)  $ is not uniform, since this would imply analyticity
at $\zeta$.
\end{proof}

\begin{corollary}
\label{rayshg}Let $z=re^{ij\pi/n},$ and $j\in%
\mathbb{Z}
.$ Then
\[
h_{n}\left(  re^{ij\pi/n}\right)  =e^{ij\pi/n}h_{n}\left(  r\right)  \text{
and }\overline{g_{n}\left(  re^{ij\pi/n}\right)  }=\left(  -1\right)
^{j}e^{ij\pi/n}g_{n}\left(  r\right)
\]
and on these rays, $\arg\left(  h_{n}\left(  z\right)  \right)  =\arg z,$ and
$\arg\left(  \overline{g_{n}\left(  z\right)  }\right)  =\arg\left(  \frac
{1}{n-1}\bar{z}^{n-1}\right)  .$
\end{corollary}

\begin{proof}
By Proposition \ref{hyperProps} (iv), $H_{n}$ and $G_{n}$ map reals to
positive reals. Thus $h_{n}$ and $g_{n}$ map positive reals to positive reals.
Now let $z=re^{ij\pi/n}$ for $r>0$ in formula (\ref{ga}).
\end{proof}

We now define the rosette harmonic mappings.

\begin{definition}
\label{defnharmonic}Let $n\in%
\mathbb{N}
$ with $n\geq3,$ and let $h_{n}$ and $g_{n}$ be defined as in Definition
\ref{defnhg}. For each $\beta\in%
\mathbb{R}
,$ define the \textbf{rosette harmonic mapping}%
\[
f_{\beta}\left(  z\right)  =e^{i\beta/2}h_{n}\left(  z\right)  +e^{-i\beta
/2}\overline{g_{n}\left(  z\right)  },\text{ }z\in\bar{U}.
\]
Then $f_{\beta}\left(  z\right)  $ is harmonic on $U,$ and continuous on
$\bar{U}.$ Denote the dilatation of $f_{\beta}\left(  z\right)  $ by
$\omega\left(  z\right)  $ on $\bar{U}\backslash\left\{  \zeta^{j}%
:j=1,2,...,2n\right\}  ,$ where $\zeta$ is a primitive 2nth root of unity.
\end{definition}

\begin{remark}
For simplicity of the notation, we do not notate the value of $n,$ which is
apparent from context, and fixed as a constant in our discussions$.$ The
dilatation is independent of $\beta$ and we simplify our notation to $\omega$
instead of using $\omega_{f_{\beta}}.$
\end{remark}

The derivatives of the analytic and co-analytic parts (for $\beta=0$) in
equation (\ref{derivatives}) are the same as the corresponding derivatives for
the hypocycloid, except for the radical factor. This factor affects the
argument of the summands in the canonical decomposition, but Corollary
\ref{rayshg}\ shows that along rays $\left\{  re^{ij\pi/n}:r>0\right\}  $
where $j\in%
\mathbb{Z}
,$ that $h_{n}$ and $\overline{g_{n}}$ are collinear. Moreover Corollary
\ref{rayshg} shows that on these same rays, that $h_{n}$ and $\bar{g}_{n}$
have the same arguments as their analytic and anti-analytic counterparts in
the hypocycloid mapping. Thus $f_{0}\left(  z\right)  $ and $f_{hyp}\left(
z\right)  $ are collinear along these radial lines (as indicated in Figure
\ref{rosette0}). Figure \ref{five} shows images $f_{\beta}\left(  U\right)  $
as $\beta$ varies$,$ and suggests that different harmonic mappings $f_{\beta}$
are obtained for the four different values of $\beta$ there. This contrasts
with the hypocycloid mapping, where rotating the analytic and anti-analytic
parts $z$ and $\frac{1}{n-1}\bar{z}^{n-1}$ relative to one another does not
yield a graph that is essentially different.

One may ask what other maps could be obtained through adding different
rotations of $h_{n}$ and $\overline{g_{n}}$. The next proposition shows that
if we consider arbitrary rotations of the analytic and co-analytic parts, by
$\theta$ and $\tilde{\theta}$ say, then the result will in fact be a rotation
of a rosette harmonic mapping.

\begin{proposition}
Let $\theta$ and $\tilde{\theta}$ be arbitrary real angles. Then $e^{i\theta
}h_{n}\left(  z\right)  +\overline{e^{i\tilde{\theta}}g_{n}\left(  z\right)
}$ is a rotation $e^{i\gamma}f_{\beta}$ for some real angles $\gamma$ and
$\beta.$
\end{proposition}

\begin{proof}
Note that the harmonic function $e^{i\theta}h_{n}\left(  z\right)
+\overline{e^{i\tilde{\theta}}g_{n}\left(  z\right)  }$ can be rewritten%
\begin{align*}
&  e^{-i\tilde{\theta}}\left(  e^{i\theta}e^{i\tilde{\theta}}h_{n}\left(
z\right)  +\overline{g_{n}\left(  z\right)  }\right) \\
&  =e^{-i\tilde{\theta}}e^{i\left(  \theta+\tilde{\theta}\right)  /2}\left(
e^{i\left(  \theta+\tilde{\theta}\right)  /2}h_{n}\left(  z\right)
+e^{-i\left(  \theta+\tilde{\theta}\right)  /2}\overline{g_{n}\left(
z\right)  }\right) \\
&  =e^{i\left(  \theta-\tilde{\theta}\right)  /2}f_{\theta+\tilde{\theta}
}\left(  z\right)
\end{align*}
Thus $e^{i\theta}h_{n}\left(  z\right)  +\overline{e^{i\tilde{\theta}}%
g_{n}\left(  z\right)  }$ can be obtained by a rotation by $\gamma=\left(
\theta-\tilde{\theta}\right)  /2$ of the map $f_{\beta}$ where $\beta=
\theta+\tilde{\theta}.$
\end{proof}

Thus the family of mappings $\left\{  f_{\beta}:\beta\in%
\mathbb{R}
\right\}  $ represents all of the different mappings, up to rotation, that
arise from arbitrary rotations of $h_{n}$ and $g_{n}$. We show in Proposition
\ref{parametrize} that $f_{\beta+\pi}$ is essentially the same mapping as
$f_{\beta}.$

\begin{proposition}
\label{parametrize}(i) For any \ $\beta\in%
\mathbb{R}
,$ the functions $f_{\beta}$ and $f_{\beta+\pi}$ are related by
\begin{equation}
f_{\beta}\left(  z\right)  =e^{-i\left(  \frac{\pi}{2}+\frac{\pi}{n}\right)
}f_{\beta+\pi}\left(  e^{i\frac{\pi}{n}}z\right)  . \label{rotbet}%
\end{equation}
Thus the image $f_{\beta+\pi}\left(  U\right)  $ is equal to a rotation of the
image $f_{\beta}\left(  U\right)  $.\newline(ii) If $\beta<0,$ then the image
of any point under $f_{\beta}$ is a reflection in the real axis of the same
point under $f_{-\beta},$ where $-\beta>0.$ \ Specifically, $f_{\beta}\left(
\bar{z}\right)  =\overline{f_{-\beta}\left(  z\right)  }.$
\end{proposition}

\begin{proof}
(i) We compute, using $j=1$ in equation (\ref{ga}),%
\[
f_{\pi+\beta}\left(  e^{i\frac{\pi}{n}}z\right)  =e^{i\left(  \pi
/2+\beta/2\right)  }e^{i\pi/n}h_{n}\left(  z\right)  +e^{-i\left(  \pi
/2+\beta/2\right)  }\overline{\left(  -1\right)  e^{-i\pi/n}g_{n}\left(
z\right)  }.
\]
Multiplying by $e^{-i\left(  \frac{\pi}{2}+\frac{\pi}{n}\right)  },$ and
noting that $-e^{-i\frac{\pi}{2}}=e^{i\frac{\pi}{2}},$ $\ $%
\[
e^{-i\left(  \frac{\pi}{2}+\frac{\pi}{n}\right)  }f_{\pi+\beta}\left(
e^{i\frac{\pi}{n}}z\right)  =e^{i\beta/2}h_{n}\left(  z\right)  +e^{-i\beta
/2}g_{n}\left(  z\right)  =f_{\beta}\left(  z\right)  .
\]
(ii) Once again, by computation:
\[
f_{\beta}\left(  \bar{z}\right)  =e^{i\frac{\beta}{2}}h_{n}\left(  \bar
{z}\right)  +e^{-i\frac{\beta}{2}}\overline{g_{n}\left(  \bar{z}\right)
}=e^{i\frac{\beta}{2}}\overline{h_{n}\left(  z\right)  }+e^{-i\frac{\beta}{2}%
}g_{n}\left(  z\right)
\]
while
\[
f_{-\beta}\left(  z\right)  =e^{-i\frac{\beta}{2}}h_{n}\left(  z\right)
+e^{i\frac{\beta}{2}}\overline{g_{n}\left(  z\right)  }.
\]
The last two expressions are conjugates of one another. $\blacksquare$
\end{proof}

\begin{corollary}
\label{transit}Let $\tilde{\beta}\in%
\mathbb{R}
,$ and let $\tilde{\beta}=\beta+l\pi$ for some $l\in%
\mathbb{Z}
.$ Then%
\begin{equation}
f_{\beta+l\pi}\left(  z\right)  =e^{il\left(  \pi/n+\pi/2\right)  }f_{\beta
}\left(  e^{-il\pi/n}z\right)  . \label{trans}%
\end{equation}

\end{corollary}

\begin{proof}
Solving equation (\ref{rotbet}) we have $f_{\beta+\pi}\left(  z\right)
=e^{i\left(  \pi/2+\pi/n\right)  }f_{\beta}\left(  e^{-i\pi/n}z\right)  .$
Equation (\ref{trans}) is obtained by repeated application of this equation.
\end{proof}

Proposition \ref{parametrize} (i) shows that up to rotations (pre and post
composed), all rosette mappings are represented in the set $\left\{  f_{\beta
}:\beta\in(-\pi/2,\pi/2]\right\}  .$ We will see in Section 4 that these
rosette mappings are all distinct from one another in that no rosette mapping
in the set can be obtained by rotations from another. Proposition
\ref{parametrize} (ii) allows us to consider just $\left\{  f_{\beta}:\beta
\in\left[  0,\pi/2\right]  \right\}  $ to obtain all rosette mappings, up to
rotation and reflection$.$
\begin{figure}[h]%
\centering
\includegraphics[
height=4.3811in,
width=4.3811in
]%
{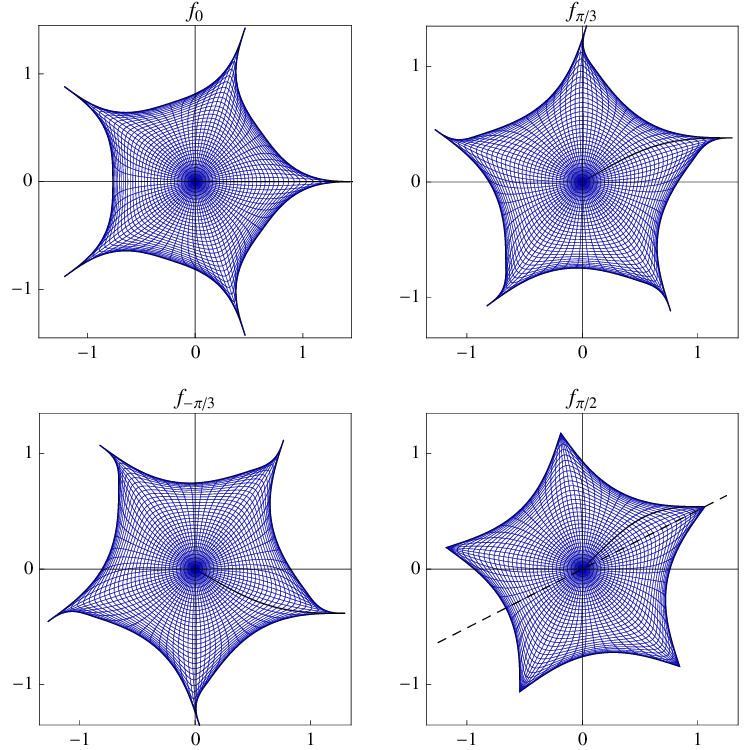}%
\caption{Images of $U$ under $f_{\beta}$ with $n=5.$ The darker "radial curve"
indicates the image of $\left[  0,1\right]  .$ The graphs $f_{\pi/3}$ and
$f_{-\pi/3}$ illustrate Proposition \ref{parametrize} (ii). The images
$f_{-\pi/3}\left(  U\right)  $ and $f_{\pi/3}\left(  U\right)  $ have cyclic
symmetry while $f_{0}\left(  U\right)  $ and $f_{\pi/2}\left(  U\right)  $
exhibit dihedral symmetry, illustrating Theorem \ref{symmetry}. For $f_{\pi
/2}$ the line of symmetry $\arg z=\pi/4-\pi/\left(  2n\right)  $ is
indicated.}%
\label{five}%
\end{figure}

Figure \ref{five} illustrates that $f_{-\pi/3}\left(  U\right)  $ and
$f_{\pi/3}\left(  U\right)  $ are reflections of one another. The example
graphs in Figure \ref{five} also demonstrate the rotational and reflectional
symmetries apparent within any particular graph $f_{\beta}\left(  U\right)  ,$
as stated in the following theorem$.$

\begin{theorem}
\label{symmetry}Let $n\in%
\mathbb{N}
$ and $n\geq3.$ \newline(i) The harmonic functions $f_{\beta}\left(  z\right)
$, $\beta\in%
\mathbb{R}
$ have dilatation $\omega\left(  z\right)  =z^{n-2}$ for $z\in U.$\newline(ii)
The rosette mapping $f_{\beta}\left(  z\right)  $, $\beta\in%
\mathbb{R}
$ has n-fold rotational symmetry, that is
\begin{equation}
f_{\beta}\left(  e^{i2k\pi/n}z\right)  =e^{i2k\pi/n}f_{\beta}\left(  z\right)
,\text{ where }z\in\bar{U},\text{ }k\in%
\mathbb{Z}
. \label{RS}%
\end{equation}
\newline(iii) If $\beta$ is an integer multiple of $\pi/2,$ then the image
$f_{\beta}\left(  U\right)  $ has reflectional symmetry$.$ For $\beta=0$ and
$\beta=\pi/2$ the reflections are%
\begin{equation}
f_{0}\left(  \overline{z}\right)  =\overline{f_{0}\left(  z\right)  }\text{
and }e^{i\eta}f_{\pi/2}\left(  e^{i\gamma}\overline{z}\right)  =\overline
{e^{i\eta}f_{\pi/2}\left(  e^{i\gamma}z\right)  }, \label{reflect}%
\end{equation}
where $\eta=\pi/\left(  2n\right)  -\pi/4$ and $\gamma=-\pi/\left(  2n\right)
.$ Thus if $z$ and $z^{\prime}$ are reflections in $\arg z=\gamma,$ then
$f_{\pi/2}\left(  z\right)  $ and $f_{\pi/2}\left(  z^{\prime}\right)  $ are
reflections in $\arg z=\eta$. \newline
\end{theorem}

\begin{proof}
(i) We compute $\omega\left(  z\right)  =$ $g_{n}^{\prime}\left(  z\right)
/h_{n}^{\prime}\left(  z\right)  $ from the derivative expressions in
(\ref{derivatives}), noting that the constant $e^{i\beta/2},$ and the
radicals, cancel leaving $z^{n-2}.$ \newline(ii) From equation (\ref{ga}) with
$j=2k$ we have $h_{n}\left(  e^{i2k\pi/n}z\right)  =e^{i2k\pi/n}h_{n}\left(
z\right)  $ and $\overline{g_{n}\left(  e^{i2k\pi/n}z\right)  }=e^{i2k\pi
/n}\overline{g_{n}\left(  z\right)  }$. Thus%
\[
f_{\beta}\left(  e^{i2k\pi/n}z\right)  =e^{i2k\pi/n}\left(  e^{i\beta/2}%
h_{n}\left(  z\right)  +e^{-i\beta/2}\overline{g_{n}\left(  z\right)
}\right)  =e^{i2k\pi/n}f_{\beta}\left(  z\right)  \text{.}%
\]
\allowbreak(iii) From Proposition \ref{analco}, $f_{0}(\overline{z}%
)=h_{n}(\bar{z})+\overline{g_{n}\left(  \bar{z}\right)  }=\overline{h_{n}%
(z)}+g_{n}\left(  z\right)  =\overline{f_{0}\left(  z\right)  }.$ For
$f_{\pi/2},$ consider the expressions $f_{\pi/2}\left(  e^{i\gamma}%
\overline{z}\right)  $ and $f_{\pi/2}\left(  e^{i\gamma}z\right)  $:\newline%
\begin{align*}
f_{\pi/2}\left(  e^{i\gamma}\overline{z}\right)   &  =e^{i\pi/4}h_{n}\left(
e^{-i\pi/\left(  2n\right)  }\bar{z}\right)  +e^{-i\pi/4}\overline
{g_{n}\left(  e^{-i\pi/\left(  2n\right)  }\bar{z}\right)  }\\
&  =e^{i\pi/4}\overline{h_{n}\left(  e^{i\pi/\left(  2n\right)  }z\right)
}+e^{-i\pi/4}g_{n}\left(  e^{i\pi/\left(  2n\right)  }z\right) \\
&  =e^{i\pi/4}\overline{e^{i\pi/\left(  2n\right)  }zH\left(  -z^{2n}\right)
}+e^{-i\pi/4}ie^{-i\pi/\left(  2n\right)  }\frac{z^{n-1}G\left(
-z^{2n}\right)  }{n-1}\\
&  =e^{i\left(  \pi/4-\pi/\left(  2n\right)  \right)  }\left(  \overline
{zH\left(  -z^{2n}\right)  }+\frac{z^{n-1}G\left(  -z^{2n}\right)  }%
{n-1}\right)  .
\end{align*}
Note that $\left(  e^{i\pi/\left(  2n\right)  }\right)  ^{n-1}=ie^{-i\pi
/\left(  2n\right)  }$ and $\left(  e^{-i\pi/\left(  2n\right)  }z\right)
^{2n}=-z^{2n},$ as used in the calculation above$.$ We also use $\left(
e^{-i\pi/\left(  2n\right)  }\right)  ^{n-1}=-ie^{i\pi/\left(  2n\right)  }%
\ $below:\newline%
\begin{align*}
f_{\pi/2}\left(  e^{i\gamma}z\right)   &  =e^{i\pi/4}h_{n}\left(
e^{-i\frac{\pi}{2n}}z\right)  +e^{-i\pi/4}\overline{g_{n}\left(
e^{-i\frac{\pi}{2n}}z\right)  }\\
&  =e^{i\pi/4}\left(  e^{-i\frac{\pi}{2n}}z\right)  H\left(  e^{-i\pi}%
z^{2n}\right)  +e^{-i\pi/4}\overline{\left(  -i\right)  e^{i\frac{\pi}{2n}%
}\frac{z^{n-1}}{n-1}G\left(  -z^{2n}\right)  }\\
&  =e^{i\left(  \pi/4-\pi/\left(  2n\right)  \right)  }\left(  zH\left(
-z^{2n}\right)  +\overline{\frac{z^{n-1}}{n-1}G\left(  -z^{2n}\right)
}\right)  .
\end{align*}
Multiplying each of $f_{\pi/2}\left(  e^{i\gamma}\overline{z}\right)  $ and
$f_{\pi/2}\left(  e^{i\gamma}z\right)  $ by $e^{i\eta}=e^{i\left(  \pi
/4-\pi/\left(  2n\right)  \right)  },$ we see that $e^{i\eta}$ $f_{\pi
/2}\left(  e^{i\gamma}\overline{z}\right)  $ and $e^{i\eta}f_{\pi/2}\left(
e^{i\gamma}z\right)  $ are conjugates of one another. $\ $Thus the stated
equations in (iii) hold.
\end{proof}

In Section 4 we use features of the boundary $\partial f_{\beta}\left(
U\right)  $ that allow us to demonstrate the precise symmetry group for each
graph $f_{\beta}\left(  U\right)  $, $\beta\in%
\mathbb{R}
$ (see Corollary \ref{noreflect}). Theorem \ref{symmetry} also shows that for
a given $n\geq3$ the rosette mappings $f_{\beta}$ and $n$-cusped hypocycloid
all have the same dilatation. The equal dilatations result in similarities in
the tangents of the rosette and hypocycloid boundary curves, to be discussed
further in Section 4.

We finish this section by laying out geometric features that are specific to
rosette mappings for $\beta$ in the interval $\beta\in(-\pi/2,\pi/2].$ These
facts in combination with equation (\ref{trans}) will allow us to extend our
conclusions for any $\beta\in%
\mathbb{R}
.$

\begin{lemma}
\label{convexconcave} Let $n\in%
\mathbb{N}
,$ $n\geq3,$ and recall $K_{n}=\sqrt{\pi}\,\Gamma\left(  1+\frac{1}%
{2n}\right)  /\Gamma\left(  \frac{1}{2}+\frac{1}{2n}\right)  .$ \newline(i)
For $\beta\in(-\pi/2,\pi/2],$ we have polar forms for $f_{\beta}\left(
1\right)  $ and $f_{\beta}\left(  e^{i\pi/n}\right)  $ given by magnitudes
$\left\vert f_{\beta}\left(  1\right)  \right\vert $ and $\left\vert f_{\beta
}\left(  e^{i\pi/n}\right)  \right\vert $ which are, respectively,
\begin{equation}
K_{n}\sqrt{\sec^{2}\left(  \frac{\pi}{2n}\right)  +2\tan\left(  \frac{\pi}%
{2n}\right)  \cos\beta}\text{ and }K_{n}\sqrt{\sec^{2}\left(  \frac{\pi}%
{2n}\right)  -2\tan\left(  \frac{\pi}{2n}\right)  \cos\beta}, \label{cnmag}%
\end{equation}
and by the arguments $\psi=\arg\left(  f_{\beta}\left(  1\right)  \right)  $
and $\pi/n+\psi^{\prime}=\arg\left(  f_{\beta}\left(  e^{i\pi/n}\right)
\right)  ,$ where%
\begin{equation}
\tan\psi=\frac{1-\tan\left(  \frac{\pi}{2n}\right)  }{1+\tan\left(  \frac{\pi
}{2n}\right)  }\tan\left(  \frac{\beta}{2}\right)  \text{ and }\tan
\psi^{\prime}=\frac{1+\tan\left(  \frac{\pi}{2n}\right)  }{1-\tan\left(
\frac{\pi}{2n}\right)  }\tan\left(  \frac{\beta}{2}\right)  . \label{cnarg}%
\end{equation}
If $\beta=0$ then both $\psi$ and $\psi^{\prime}$ are zero. If $\beta=\pi/2,$
these angles reduce to \linebreak$\psi=\pi/4-\pi/\left(  2n\right)  $ and
$\psi^{\prime}=\pi/4+\pi/\left(  2n\right)  .$\newline(ii) If $\beta\in
(0,\pi/2],$ then the curves $f_{\beta}\left(  r\right)  $ and $f_{\beta
}\left(  re^{i\pi/n}\right)  $ have strictly increasing magnitude. Moreover,
$\arg\frac{\partial}{\partial r}f_{\beta}\left(  r\right)  $ decreases
strictly with $r$, and $\arg\frac{\partial}{\partial r}f_{\beta}\left(
re^{i\pi/n}\right)  ,$ increases strictly with $r$. We also have the following
arguments for the the tangents of these curves at the origin
\[
\lim_{r\rightarrow0^{+}}\arg\frac{\partial}{\partial r}f_{\beta}\left(
r\right)  =\beta/2\text{ and }\lim_{r\rightarrow0^{+}}\arg\frac{\partial
}{\partial r}f_{\beta}\left(  re^{i\pi/n}\right)  =\beta/2+\pi/n,
\]
and at the boundary of $U$
\[
\lim_{r\rightarrow1^{-}}\arg\frac{\partial}{\partial r}f_{\beta}\left(
r\right)  =0\text{ and }\lim_{r\rightarrow1^{-}}\arg\frac{\partial}{\partial
r}f_{\beta}\left(  re^{i\pi/n}\right)  =\pi/2+\pi/n.
\]
\newline(iii) For $\beta=0$ the curves $f_{0}\left(  r\right)  $ and
$f_{0}\left(  re^{i\pi/n}\right)  $ are straight line-segments - the images
have constant argument of $0$ and $\pi/n,$ respectively.
\end{lemma}

\begin{proof}
To prove (i), we compute
\begin{align*}
f_{\beta}\left(  r\right)   &  =e^{i\beta/2}h_{n}\left(  r\right)
+e^{-i\beta/2}g_{n}\left(  r\right) \\
&  =\cos\left(  \beta/2\right)  \left(  h_{n}\left(  r\right)  +g_{n}\left(
r\right)  \right)  +i\sin\left(  \beta/2\right)  \left(  h_{n}\left(
r\right)  -g_{n}\left(  r\right)  \right)  ,
\end{align*}
and, upon recalling from Corollary \ref{rayshg} that $\overline{g_{n}\left(
re^{i\pi/n}\right)  }=e^{-i\pi/n}g_{n}\left(  r\right)  ,$
\begin{align*}
f_{\beta}\left(  re^{i\pi/n}\right)   &  =e^{i\beta/2}h_{n}\left(  re^{i\pi
/n}\right)  +e^{-i\beta/2}\overline{e^{-i\pi/n}g_{n}\left(  r\right)  }\\
&  =e^{i\pi/n}\left(  \cos\left(  \beta/2\right)  \left(  h_{n}\left(
r\right)  -g_{n}\left(  r\right)  \right)  +i\sin\left(  \beta/2\right)
\left(  h_{n}\left(  r\right)  +g_{n}\left(  r\right)  \right)  \right)  .
\end{align*}
We recall the formulae $h_{n}\left(  1\right)  =H_{n}\left(  1\right)  =K_{n}$
and $g_{n}\left(  1\right)  =G_{n}\left(  1\right)  /\left(  n-1\right)
=\tan\left(  \pi/\left(  2n\right)  \right)  K_{n}$ from Proposition
\ref{hyperProps}. Using continuity of $f_{\beta}$ on $\bar{U}$, we take limits
as $r\rightarrow1^{-}$ to get%
\begin{align*}
f_{\beta}\left(  1\right)   &  =K_{n}\left(  \cos\left(  \beta/2\right)
\left(  1+\tan\left(  \frac{\pi}{2n}\right)  \right)  +i\sin\left(
\beta/2\right)  \left(  1-\tan\left(  \frac{\pi}{2n}\right)  \right)  \right)
\text{ and}\\
f_{\beta}\left(  e^{i\pi/n}\right)   &  =e^{i\pi/n}K_{n}\left(  \cos\left(
\beta/2\right)  \left(  1-\tan\left(  \frac{\pi}{2n}\right)  \right)
+i\sin\left(  \beta/2\right)  \left(  1+\tan\left(  \frac{\pi}{2n}\right)
\right)  \right)  .
\end{align*}
Moreover, writing $\psi=\arg f_{\beta}\left(  1\right)  ,$ and $\pi
/n+\psi^{\prime}=\arg f_{\beta}\left(  e^{i\pi/n}\right)  ,$ we take ratios of
imaginary to real parts of $f_{\beta}\left(  1\right)  $ and $f_{\beta}\left(
e^{i\pi/n}\right)  $ to easily obtain the stated formulae for $\tan\psi$ and
$\tan\psi^{\prime}.$ We also readily obtain
\begin{align*}
\left\vert f_{\beta}\left(  1\right)  \right\vert ^{2}  &  =K_{n}^{2}\left(
1+2\tan\left(  \frac{\pi}{2n}\right)  \cos\beta+\tan^{2}\left(  \frac{\pi}%
{2n}\right)  \right)  \text{ and }\\
\left\vert f_{\beta}\left(  e^{i\pi/n}\right)  \right\vert ^{2}  &  =K_{n}%
^{2}\left(  1-2\tan\left(  \frac{\pi}{2n}\right)  \cos\beta+\tan^{2}\left(
\frac{\pi}{2n}\right)  \right)
\end{align*}
We can rewrite $1+\tan^{2}\left(  \pi/\left(  2n\right)  \right)  =\sec
^{2}\left(  \pi/\left(  2n\right)  \right)  $ to obtain the magnitudes stated
in (i)$.$ For (ii), taking derivatives of the formulae first derived for
$f_{\beta}\left(  r\right)  $ and $f_{\beta}\left(  re^{i\pi/n}\right)  ,$ we
obtain
\begin{align*}
\frac{\partial}{\partial r}f_{\beta}\left(  r\right)   &  =\cos\left(
\beta/2\right)  \left(  \frac{1+r^{n-2}}{\sqrt{1-r^{2n}}}\right)
+i\sin\left(  \beta/2\right)  \left(  \frac{1-r^{n-2}}{\sqrt{1-r^{2n}}%
}\right)  \text{ and }\\
\frac{\partial}{\partial r}f_{\beta}\left(  re^{i\pi/n}\right)   &
=e^{i\pi/n}\left(  \cos\left(  \beta/2\right)  \left(  \frac{1-r^{n-2}}%
{\sqrt{1-r^{2n}}}\right)  +i\sin\left(  \beta/2\right)  \left(  \frac
{1+r^{n-2}}{\sqrt{1-r^{2n}}}\right)  \right)  .
\end{align*}

The arguments of these derivatives, namely $\arg\frac{\partial}{\partial
r}f_{\beta}\left(  r\right)  $ and $\arg\frac{\partial}{\partial r}f_{\beta
}\left(  re^{i\pi/n}\right)  $ are respectively\
\[
\arctan\left(  \tan\left(  \beta/2\right)  \frac{1-r^{n-2}}{1+r^{n-2}}\right)
\text{ and }\pi/n+\arctan\left(  \tan\left(  \beta/2\right)  \frac{1+r^{n-2}%
}{1-r^{n-2}}\right)  ,
\]
whence the stated monotonicity of $\arg\frac{\partial}{\partial r}f_{\beta
}\left(  r\right)  $ and $\arg\frac{\partial}{\partial r}f_{\beta}\left(
re^{i\pi/n}\right)  $ in (ii).\linebreak The limits in (ii) are now easily
computed. Note that as $r\rightarrow1^{-},$ the ratio $\frac{1+r^{n-2}%
}{1-r^{n-2}}$ becomes infinite, and $\arctan\left(  \tan\left(  \beta
/2\right)  \frac{1+r^{n-2}}{1-r^{n-2}}\right)  $ approaches $\pi/2.\ $We can
also calculate the magnitudes $\left\vert \frac{\partial}{\partial r}f_{\beta
}\left(  r\right)  \right\vert ^{2}$ and $\left\vert \frac{\partial}{\partial
r}f_{\beta}\left(  re^{i\pi/n}\right)  \right\vert ^{2},$ respectively, as
$\frac{1\pm2r^{n-2}\cos\beta+r^{2n-4}}{1-r^{2n}}>\frac{\left(  1-r^{n-2}%
\right)  ^{2}}{1-r^{2n}}>0.$ Thus both curves $f_{\beta}\left(  r\right)  $
and $f_{\beta}\left(  re^{i\pi/n}\right)  $ have increasing magnitude as $r$
increases. Finally to prove (iii), where $\beta=0,$ we have \newline%
$f_{0}\left(  r\right)  =\left(  1/\sqrt{2}\right)  \left(  h_{n}\left(
r\right)  +g_{n}\left(  r\right)  \right)  >0$. Moreover, $\arg f_{0}\left(
re^{i\pi/n}\right)  =\frac{\pi}{n}+\arctan\left(  0\right)  =\frac{\pi}{n},$
so $f_{0}\left(  re^{i\frac{\pi}{n}}\right)  $ maps to a ray emanating from
the origin with argument $\pi/n.$
\end{proof}

\section{Cusps and Nodes}

We now examine the boundary curves for the rosette harmonic mappings, which
allows us to describe the cusps and other features that are apparent in the
graphs$.$ The boundary curve is also the key in our approach in Section 5 to
proving univalence of the rosette harmonic mappings.

For a fixed $n\geq3$, a rosette harmonic mapping $f_{\beta}$ of Definition
\ref{defnharmonic} extends analytically to $\partial U$, except at isolated
values on $\partial U$. Indeed recall that $h_{n}$ and $g_{n}$ of Definition
\ref{defnhg} are analytic except at the $2n$th roots of unity (Proposition
\ref{analco}). In contrast, the $n$-cusped hypocycloid $f_{hyp}$ of Example
1.1 extends continuously to $\bar{U}$ but with just $n$ values in $\partial U
$ where the boundary function is not regular.

Nevertheless, Figures \ref{rosette0}\ and \ref{five} show rosette mappings
$f_{\beta},$ where exactly $n$ cusps are apparent. A striking similarity
between the rosette and hypocycloid mappings is the common dilatation
$\omega\left(  z\right)  =z^{n-2}.$ We note the following consequence for a
harmonic function $f$ on $U.$ Where $\alpha\left(  t\right)  =f\left(
e^{it}\right)  $ exists with a continuous derivative $\alpha^{\prime}\left(
t\right)  $, Corollary 2.2b of \cite{HenSchob} implies that $\operatorname{Im}%
\left(  \sqrt{\omega_{f}\left(  e^{it}\right)  }\alpha^{\prime}\left(
t\right)  \right)  =0$ (see also Section 7.4 of \cite{PeterHMBook}). Thus on
intervals where $\alpha^{\prime}\left(  t\right)  $ is continuous and
non-zero, and for dilatation $\omega_{f}\left(  z\right)  =z^{n-2}$, we have
\begin{equation}
\arg\alpha^{\prime}\left(  t\right)  \equiv-\arg\sqrt{e^{i\left(  n-2\right)
t}}\equiv-\left(  n/2-1\right)  t\text{ }(\operatorname{mod}\pi).
\label{HSarg}%
\end{equation}
We will find for rosette mappings $f_{\beta}$ of Definition \ref{defnharmonic}
that have cusps, $n$ of the $2n$ singular points are ``removable" in a sense
to be described. Furthermore, for $n\geq3$ and consistent with (\ref{HSarg}),
the formula
\begin{equation}
\fbox{$\arg\alpha^\prime\left(  t\right)  =k\pi-\left(  \frac{n}{2}-1\right)
t$} \label{compass}%
\end{equation}
holds (except at possibly one point) on each of the $n$ intervals $\left(
\left(  2k-2\right)  \pi/n,2k\pi/n\right)  $, $\ k=1,2,...,n,$ both when
$\alpha$ is the boundary function of an $n$-cusped hypocycloid \textit{and
}when $\alpha$ is the boundary function of a rosette (provided it has cusps).
Formula (\ref{compass}) will not be valid for example for a rosette mapping
$f_{\pi/2}$, where there are arcs for which the boundary function is
constant$.$ To proceed we first give a definition of cusp, node, and singular
point of a curve.

\begin{definition}
\label{cusp}An \textbf{isolated singular point} on a curve $\alpha\left(
t\right)  $ is a point $\alpha\left(  t_{0}\right)  $ at which $\alpha
^{\prime}\left(  t\right)  $ is defined and non-zero in a punctured
neighborhood of $t_{0},$ but either (i) $\alpha^{\prime}\left(  t_{0}\right)
=0$ or (ii) $\alpha^{\prime}\left(  t_{0}\right)  $ is not defined. Define the
quantities
\[
\arg\alpha^{\prime}\left(  t_{0}\right)  ^{-}=\lim_{t\nearrow t_{0}^{-}}%
\arg\alpha^{\prime}\left(  t\right)  \text{ and }\arg\alpha^{\prime}\left(
t_{0}\right)  ^{+}=\lim_{t\searrow t_{0}^{+}}\arg\alpha^{\prime}\left(
t\right)
\]
where they exist. An isolated singular point for which $\arg\alpha^{\prime
}\left(  t_{0}\right)  ^{+}$and $\arg\alpha^{\prime}\left(  t_{0}\right)
^{-}$ differ by $\pi$ is defined to be a \textbf{cusp.} The line $L$ through
the cusp $\alpha\left(  t_{0}\right)  $ containing points with argument equal
to $\arg\alpha^{\prime}\left(  t_{0}\right)  ^{+}$ (or $\arg\alpha^{\prime
}\left(  t_{0}\right)  ^{-}$) is called the \textbf{axis of the cusp}, or
simply the \textbf{axis}$.$ Define a \textbf{node} to be an isolated singular
point on the curve at which $\arg\alpha^{\prime}\left(  t_{0}\right)
^{+}-\arg\alpha^{\prime}\left(  t_{0}\right)  ^{-}=\theta\not \equiv
\pi\,\left(  \operatorname{mod}2\pi\right)  .$ In this case, $\theta$ is the
\textbf{exterior angle} of the node. The interior angle at the node is then
$\pi-\theta.$ If the exterior angle is $0,$ then we call the node a
\textbf{removable node}. A node is described as a \textbf{corner} in some sources.
\end{definition}

Both $h_{n}$ and $g_{n}$ have nodes with exterior (and interior) angle
$\pi/2,$ as seen in Figure \ref{canon6}: since $\overline{g_{n}}$ is a
reflection of $g_{n}$, the nodes of $\overline{g_{6}}$ in Figure \ref{canon6}
appear with exterior angle $-\pi/2$ rather than $\pi/2$. The lower right image
in Figure \ref{five} indicates a rosette mapping with nodes rather than cusps,
but the remaining images in Figure \ref{five} show examples with cusps. The
following lemma provides a convenient way to invoke equation (\ref{compass})
and draw conclusions about cusps of the boundary extension.

\begin{lemma}
\label{cusps}Let $n\geq3,$ and $f\left(  z\right)  $ be a harmonic mapping on
$U$, with continuous extension to $\bar{U},$ so that $\alpha\left(  t\right)
=f\left(  e^{it}\right)  $ is defined on $\partial U$. Let $k=1,2,...,n.$
\newline(i)\ Suppose that $\alpha^{\prime}$ is defined and non-zero except at
$t=2k\pi/n,$ and that $\alpha$ satisfies (\ref{compass}) on each interval
$\left(  \left(  2k-2\right)  \pi/n,2k\pi/n\right)  .$ Then for each $k,$
$\alpha\left(  2k\pi/n\right)  $ is a cusp$,$ and the cusp axis has argument
$2k\pi/n.$\newline(ii)\ Suppose that $f$ has $n$-fold rotational symmetry
$f\left(  e^{i2k\pi/n}z\right)  =e^{i2k\pi/n}f\left(  z\right)  ,$ and
that\ (\ref{compass}) holds for $k=1$ on the interval $\left(  0,2\pi
/n\right)  .$ Then $\alpha$ satisfies (\ref{compass}) on each interval
$\left(  \left(  2k-2\right)  \pi/n,2k\pi/n\right)  $ and the conclusions of
(i) hold.
\end{lemma}

\begin{proof}
For (i), we evaluate limits at $t=\frac{2k\pi}{n}$ as follows using
(\ref{compass}):
\begin{align}
\arg\alpha^{\prime}\left(  2k\pi/n\right)  ^{-}  &  =\lim_{t\nearrow
2k\pi/n^{-}}k\pi-\left(  \frac{n}{2}-1\right)  t=\frac{2k\pi}{n},\text{
and}\label{cuspslopes}\\
\arg\alpha^{\prime}\left(  2k\pi/n\right)  ^{+}  &  =\lim_{t\searrow
2k\pi/n^{+}}\left(  k+1\right)  \pi-\left(  \frac{n}{2}-1\right)  t=\pi
+\frac{2k\pi}{n}.\nonumber
\end{align}
Thus $\arg\alpha^{\prime}\left(  2\pi/n\right)  ^{+}$ and $\arg\alpha^{\prime
}\left(  2\pi/n\right)  ^{-}$ differ by $\pi,$ so $\alpha\left(
2k\pi/n\right)  $ is a cusp. We also see the axis has argument $2k\pi/n$ (or
equivalently $\pi+2k\pi/n$). For (ii), we use the fact that $\alpha\left(
t+2k\pi/n\right)  =e^{i2k\pi/n}\alpha\left(  t\right)  $ to conclude
\begin{equation}
\arg\alpha\left(  t+2k\pi/n\right)  =2k\pi/n+\arg\alpha\left(  t\right)  .
\label{badd}%
\end{equation}
Given that (\ref{compass}) holds for $k=1,$ we can extend it to each interval
$\left(  \left(  2k-2\right)  \pi/n,2k\pi/n\right)  ,$ $k=2,3,...,n,$ using
(\ref{badd}), and so the conclusions of (i)\ hold also.
\end{proof}

\begin{remark}
\label{extend}Lemma \ref{cusps} remains valid (with appropriate adjustments to
the interval on which (\ref{compass}) holds), even if for finitely many points
in $\left(  \left(  2k-2\right)  \pi/n,2k\pi/n\right)  ,$ $\arg\alpha^{\prime
}\left(  t\right)  $ does not exist. To compute the limits (\ref{cuspslopes})
we only need equation (\ref{compass}) to hold in a neighborhood of the
endpoints $2k\pi/n$.
\end{remark}

The following proposition surely appears in the literature, but for
completeness, we use Lemma \ref{cusps} to demonstrate the properties of the
hypocycloid cusps.

\begin{proposition}
\label{hypocycloid}Let $n\geq3,$ and let $f_{hyp}\left(  z\right)  =z+\frac
{1}{n-1}\bar{z}^{n-1}$ be the hypocycloid harmonic mapping, and let
$\alpha_{hyp}\left(  t\right)  =f_{hyp}\left(  e^{it}\right)  .$ For
$k=1,2,...,n,$ formula (\ref{compass}) holds with $\alpha=\alpha_{hyp}$ for
all $t\in\left(  \left(  2k-2\right)  \pi/n,2k\pi/n\right)  .$ Moreover,
$\alpha_{hyp}$ has precisely $n$ cusps $\alpha_{hyp}\left(  2k\pi/n\right)
=\frac{n}{n-1}e^{i2k\pi/n};$ the cusp axis has argument $2k\pi/n.$ In
traversing from one cusp to the next, $\alpha_{hyp}$ has total curvature
$\pi-2\pi/n.$
\end{proposition}

\begin{proof}
We have $\alpha_{hyp}\left(  t\right)  =h\left(  e^{it}\right)  +\bar
{g}\left(  e^{it}\right)  $ where $h\left(  e^{it}\right)  =e^{it}$ and
$g\left(  e^{it}\right)  =\frac{1}{n-1}e^{i\left(  n-1\right)  t}.$ We already
noted the singular points at $2k\pi/n$ in Example 1.1. With $z=e^{it}$ we
apply the chain rule, obtaining $\frac{d}{dt}h\left(  e^{it}\right)  =ie^{it}$
and $\frac{d}{dt}g\left(  e^{it}\right)  =ie^{i\left(  n-1\right)  t}.$ The
magnitudes are both $1,$ so $\alpha_{hyp}^{\prime}\left(  t\right)  $ will
have argument equal to the mean of $\arg\frac{d}{dt}h\left(  e^{it}\right)  $
and $\arg\frac{d}{dt}g\left(  e^{-it}\right)  ,$ when (for example) we choose
branches of the arguments that lie within $\pi$ of one another. To this end,
we take
\begin{equation}
\arg\frac{d}{dt}h\left(  e^{it}\right)  =\pi/2+t\text{ and }\arg\frac{d}%
{dt}g\left(  e^{-it}\right)  =3\pi/2-\left(  n-1\right)  t \label{hypderivs}%
\end{equation}
on the interval $\left(  0,2\pi/n\right)  .$ Thus the mean is $\pi-\left(
n/2-1\right)  t,$ which is equation (\ref{compass}) for $k=1$. We also have
$f_{hyp}\left(  e^{i2k\pi/n}z\right)  =e^{i2k\pi/n}z+\frac{1}{n-1}\left(
e^{-i2k\pi/n}\bar{z}\right)  ^{n-1}.$ We factor out the $e^{i2k\pi/n},$ using
$\left(  e^{-i2k\pi/n}\right)  ^{n-1}=e^{-i2k\pi/n},$ and obtain $e^{i2k\pi
/n}f_{hyp}\left(  z\right)  ,$ showing that $f_{hyp}$ has rotational
symmetry$.$ By Lemma \ref{cusps} (ii), $\alpha_{hyp}\left(  2k\pi/n\right)  $
is a cusp and the axis has argument $2k\pi/n$ for each $k=1,2,...,n.$ Using
$z=1$ above, we also obtain $f_{hyp}\left(  e^{i2k\pi/n}\right)  =e^{i2k\pi
/n}+\frac{1}{n-1}\left(  e^{i2k\pi/n}\right)  =\frac{n}{n-1}e^{i2k\pi/n}$. The
total curvature of $\alpha_{hyp}$ is measured with the change in argument of
the unit tangent, or equivalently the change in $\arg\alpha_{hyp}^{\prime}.$
Since this is monotonic and linear in $t$ (equation (\ref{compass})), the
total change in $\arg\alpha_{hyp}^{\prime}$ over any of the given intervals is
equal to $\left(  n/2-1\right)  $ times the interval length $2\pi/n$, and so
we obtain $\pi-2\pi/n$.
\end{proof}

We compute formulae for the derivatives $\left(  d/dt\right)  h_{n}\left(
e^{it}\right)  $ and $\left(  d/dt\right)  g_{n}\left(  e^{it}\right)  .$ In
contrast with the hypocycloid, the arguments $\left(  d/dt\right)
h_{n}\left(  e^{it}\right)  $ and $\left(  d/dt\right)  \overline{g_{n}\left(
e^{it}\right)  } $ differ by a constant angle of $\pm\pi/2.$

\begin{proposition}
\label{linearangles}Let $n\geq3,$ and let $\zeta$ be a primitive 2nth root of
unity, and consider $h_{n}$ and $g_{n}$ defined in Definition \ref{defnhg},
analytic on $\bar{U}\backslash\left\{  \zeta^{j}:j=1,2,...,2n\right\}  .$ Let
$j=1,2,...,2n.$\newline i) On each interval $\left(  \left(  j-1\right)
\pi/n,j\pi/n\right)  ,\ $derivatives $d/dt$ of both $h_{n}\left(
e^{it}\right)  $ and $g_{n}\left(  e^{it}\right)  $ have magnitude
$1/\left\vert \sqrt{1-e^{i2nt}}\right\vert ,$ and the arguments are linear
monotonic functions, expressible as%
\begin{align}
\arg\left(  \frac{d}{dt}h_{n}\left(  e^{it}\right)  \right)   &
=3\pi/4-\left(  n/2-1\right)  t+\left(  j-1\right)  \pi/2,\text{
and}\label{argghh}\\
\arg\left(  \frac{d}{dt}g_{n}\left(  e^{it}\right)  \right)   &
=3\pi/4+\left(  n/2-1\right)  t+\left(  j-1\right)  \pi/2. \label{argghg}%
\end{align}
\newline(ii) The functions $h_{n}\left(  e^{it}\right)  $ and $g_{n}\left(
e^{it}\right)  $ each have $2n$ singular points $h_{n}\left(  e^{ij\pi
/n}\right)  $ and $g_{n}\left(  e^{ij\pi/n}\right)  ,$ which are each nodes
with exterior (and interior) angle $\pi/2.$\newline(iii) The difference in the
arguments $\arg\frac{d}{dt}h_{n}\left(  e^{it}\right)  -\arg\frac{d}%
{dt}\overline{g_{n}}\left(  e^{it}\right)  $ is constant on $\left(  \left(
j-1\right)  \pi/n,j\pi/n\right)  ,$ and is alternately $+\pi/2$ when $j$ is
even, and $-\pi/2$ when $j$ is odd$.$
\end{proposition}

\begin{proof}
(i) Recall the derivatives $\frac{dh}{dz}=\frac{1}{\sqrt{1-z^{2n}}}$ and
$\frac{dg}{dz}=\frac{z^{n-2}}{\sqrt{1-z^{2n}}}$ of Proposition \ref{analco}
(iii). \ We have $\frac{d}{dt}h_{n}\left(  e^{it}\right)  =\frac{ie^{it}%
}{\sqrt{1-e^{i2nt}}}$ and $\frac{d}{dt}g_{n}\left(  e^{it}\right)
=\frac{ie^{i\left(  n-1\right)  t}}{\sqrt{1-e^{i2nt}}}.$ Each derivative has
magnitude $1/\left\vert \sqrt{1-e^{i2nt}}\right\vert ,$ and singular points
occur at the 2nth roots of unity, where $e^{i2nt}=1.$ We compute
\begin{align*}
\arg\frac{d}{dt}h_{n}\left(  e^{it}\right)   &  =\arg ie^{it}-\frac{1}{2}%
\arg\left(  1-e^{i2nt}\right) \\
&  =\pi/2+t+\frac{1}{2}\arctan\frac{\sin\left(  2nt\right)  }{1-\cos\left(
2nt\right)  }.
\end{align*}
The latter term reduces to $\pi/4-\left(  n/2\right)  t,$ which can be seen
for example using the half angle formula for cotangent; $\operatorname{arccot}%
\frac{\sin\left(  2nt\right)  }{1-\cos\left(  2nt\right)  }=nt$ and $\arctan
X=\pi/2-\operatorname{arccot}X.$ Thus on $\left(  0,\pi/n\right)  ,$%
\[
\arg\frac{d}{dt}h_{n}\left(  e^{it}\right)  =3\pi/4-\left(  n/2-1\right)  t.
\]
At $t=\pi/\left(  2n\right)  ,$ $\sqrt{1-e^{i2nt}}$ becomes real and
$\arg\frac{d}{dt}h_{n}\left(  e^{i\pi/2n}\right)  =\arg\left(  ie^{i\pi
/2n}\right)  =\pi/2+\pi/\left(  2n\right)  ,$ which is consistent with our
formula on $\left(  0,\pi/n\right)  $; this choice of branch of $\arctan$
gives $\arg\frac{d}{dt}h_{n}\left(  e^{it}\right)  $ an ``initial" value
$3\pi/4$ (the limit as $t\searrow0^{+}$) and is evidently consistent with the
argument at $t=e^{i\pi/2n}$ (see also Figure \ref{canon6}). From the
rotational symmetry equation (\ref{ga}) for $h_{n}$%
\[
\arg\frac{d}{dt}h_{n}\left(  e^{i\left(  t+j\pi/n\right)  }\right)  =\arg
\frac{d}{dt}h_{n}\left(  e^{it}\right)  +j\pi/n.
\]
Adding $\left(  j-1\right)  \pi/2$ extends our formula for $\arg\frac{d}%
{dt}h_{n}\left(  e^{it}\right)  $ from $\left(  0,\pi/n\right)  $ to the
interval $\left(  \left(  j-1\right)  \pi/n,j\pi/n\right)  ,$ giving
(\ref{argghh}). Similarly on $\left(  0,\pi/n\right)  $ we obtain
\[
\arg\frac{d}{dt}g_{n}\left(  e^{it}\right)  =\pi/2+\left(  n-1\right)
t+\frac{1}{2}\arctan\frac{\sin\left(  2nt\right)  }{1-\cos\left(  2nt\right)
}=3\pi/4+\left(  n/2-1\right)  t,
\]
a branch of the argument for which $\arg\frac{d}{dt}g_{n}\left(  e^{i\pi
/2n}\right)  =\pi-\pi/\left(  2n\right)  $ as expected$,$ so again with
initial value $3\pi/4$ on $\left(  0,\pi/n\right)  .$ From the rotational
symmetry equation (\ref{ga}) for $g_{n}$ we have $g_{n}\left(  e^{ij\pi
/n}z\right)  =e^{ij\left(  \pi-\pi/n\right)  }g_{n}\left(  z\right)  ,$ so
\[
\arg\frac{d}{dt}g_{n}\left(  e^{i\left(  t+j\pi/n\right)  }\right)  =\arg
\frac{d}{dt}g_{n}\left(  e^{it}\right)  +j\left(  \pi-\pi/n\right)  .
\]
We extend our formula for $\arg\frac{d}{dt}g_{n}\left(  e^{it}\right)  $ for
$t\in$ $\left(  \left(  j-1\right)  \pi/n,j\pi/n\right)  $ as before, adding
$\left(  j-1\right)  \pi/2$, leading to equation (\ref{argghg}). For (ii) we
note that as we pass from the interval $\left(  \left(  j-1\right)  \pi
/n,j\pi/n\right)  $ to $\left(  j\pi/n,\left(  j+1\right)  \pi/n\right)  ,$
both $\arg\frac{d}{dt}h_{n}\left(  e^{it}\right)  $ and $\arg\frac{d}{dt}%
g_{n}\left(  e^{it}\right)  $ increase by $\pi/2$ at $j\pi/n.$ Thus
$h_{n}\left(  e^{ij\pi/n}\right)  $ and $g_{n}\left(  e^{ij\pi/n}\right)  $
each are nodes with exterior angle $\pi/2.$ To prove (iii) we compute the
difference in $\arg\frac{d}{dt}h_{n}\left(  e^{it}\right)  $ and $\arg\frac
{d}{dt}\overline{g_{n}}\left(  e^{it}\right)  $ as $\arg\frac{d}{dt}%
h_{n}\left(  e^{it}\right)  +\arg\frac{d}{dt}g_{n}\left(  e^{it}\right)  ,$ so
for \newline$t\in\left(  \left(  j-1\right)  \pi/n,j\pi/n\right)  ,$
\begin{equation}
\arg\frac{d}{dt}h_{n}\left(  e^{it}\right)  -\arg\frac{d}{dt}\overline{g_{n}%
}\left(  e^{it}\right)  =3\pi/2+\left(  j-1\right)  \pi=\left(  2j+1\right)
\pi/2. \label{hgdiff}%
\end{equation}

\end{proof}

\begin{remark}
The rosette mappings are distinguished from the hypocycloid in that $\arg$
$\frac{d}{dt}h_{n}\left(  e^{it}\right)  $ and $\arg\frac{d}{dt}%
\overline{g_{n}\left(  e^{it}\right)  }$ are decreasing in lockstep. As a
result the curves $h_{n}\left(  e^{it}\right)  $ and $\overline{g_{n}}\left(
e^{it}\right)  ,$ and ultimately $g_{n}\left(  e^{it}\right)  $ must be rigid
motions of one another, as illustrated in Figure \ref{canon6}, and Figure
\ref{transmirror}, and proved in Corollary \ref{rigid}. For the hypocycloid,
the arguments of the derivatives of the analytic and anti-analytic parts
(\ref{hypderivs}) are also linear, but with non-equal slopes with differing sign.
\end{remark}

\begin{corollary}
\label{rigid}The graph $h_{n}\left(  e^{it}\right)  $ on an interval $\left(
\left(  j-1\right)  \pi/n,j\pi/n\right)  $ and the graph of $g_{n}\left(
e^{it}\right)  $ on an interval $\left(  \left(  j^{\prime}-1\right)
\pi/n,j^{\prime}\pi/n\right)  $ are identical, up to a translation and
rotation$,$ where $j,j^{\prime}=1,2,...,2n.$ Moreover the two curves have
opposite orientation.
\end{corollary}

\begin{proof}
The tangents $\frac{d}{dt}h_{n}\left(  e^{it}\right)  $ and $\frac{d}%
{dt}\overline{g_{n}}\left(  e^{it}\right)  $ have equal magnitudes on
\linebreak$\left(  \left(  j-1\right)  \pi/n,j\pi/n\right)  $, where their
arguments differ by a constant. Thus $h_{n}\left(  e^{it}\right)  $ and
$\overline{g_{n}\left(  e^{it}\right)  }$ have equal arclength and curvature,
and so are equal up to a translation and rotation by the fundamental theorem
of plane curves. Both $h_{n}\left(  e^{it}\right)  $ and $g_{n}\left(
e^{it}\right)  $ have rotational symmetry (Proposition \ref{analco}) so the
previous statement is true even when $h_{n}\left(  e^{it}\right)  $ and
$g_{n}\left(  e^{it}\right)  $ are defined on different arcs. Proposition
\ref{analco} also shows that the curve $h_{n}\left(  e^{int}\right)  $ also
has reflectional symmetry, so the graph of $h_{n}$ has symmetry group
$D_{2n}.$ Thus $\overline{h_{n}\left(  e^{it}\right)  }$ is also a rotation of
$h_{n}\left(  e^{it}\right)  $ on any interval $\left(  \left(  j-1\right)
\pi/n,j\pi/n\right)  ,$ where the pair has opposite orientation. We conclude
$\overline{h_{n}\left(  e^{it}\right)  }$ and $\overline{g_{n}\left(
e^{it}\right)  }$ are rigid motions of one another with opposite orientation,
and the Corollary follows upon conjugation.
\end{proof}

Proposition \ref{linearangles} allows us to compute the derivative of the
boundary function of a rosette harmonic mapping. The cosine rule for triangles
is useful for adding numbers of the same magnitude, and we recall its
application in the following remark.

\begin{remark}
\label{cosine} For $X>0,$ the cosine rule yields
\[
\left\vert Xe^{i\theta_{1}}+Xe^{i\theta_{2}}\right\vert =X\left\vert
e^{i\theta_{1}}+e^{i\theta_{2}}\right\vert =\sqrt{2}X \sqrt{ 1+\cos\left\vert
\theta_{1}-\theta_{2}\right\vert } .
\]

\end{remark}

\begin{theorem}
\label{fbdry}For $n\geq3$ and $\beta\in(-\pi/2,\pi/2],$ and let $f_{\beta}$ be
a rosette harmonic mapping defined in Definition \ref{defnharmonic}. Consider
the boundary curve $\alpha_{\beta}\left(  t\right)  =f_{\beta}\left(
e^{it}\right)  ,$ $t\in\partial U.$ The derivative $\alpha_{\beta}^{\prime
}\left(  t\right)  $ exists and is continuous on $\partial U,$ except at the
$2n$ multiples of $\pi/n.$ Let $k=1,2,...,n.$ \newline(i) For $\left\vert
\beta\right\vert <\pi/2,$ $\alpha_{\beta}$ satisfies (\ref{compass}) on
$\left(  \left(  2k-2\right)  \pi/n,2k\pi/n\right)  ,$ except at $t=\left(
2k-1\right)  \pi/n$ where $\alpha_{\beta}^{\prime}\left(  t\right)  $ is
undefined. Moreover the magnitude of $\alpha_{\beta}^{\prime}\left(  t\right)
$, which is strictly non-zero, is%
\begin{equation}
\left\vert \alpha_{\beta}^{\prime}\left(  t\right)  \right\vert =\sqrt{2}%
\sqrt{1\pm\sin\left(  \beta\right)  }/\left\vert \sqrt{1-e^{i2nt}}\right\vert
. \label{mag}%
\end{equation}
Here the $\sin$ term is subtracted on the first half of the interval $\left(
\left(  2k-2\right)  \pi/n,2k\pi/n\right)  ,$ and added on the second
half.\newline(ii) When $\beta=\pi/2,$ on the first half of the interval
$\left(  \left(  2k-2\right)  \pi/n,2k\pi/n\right)  ,$ $\alpha_{\pi/2}$
satisfies (\ref{compass})$,$ while $\alpha_{\pi/2}^{\prime}\left(  t\right)  $
is strictly non-zero there with $\left\vert \alpha_{\pi/2}^{\prime}\left(
t\right)  \right\vert =2/\left\vert \sqrt{1-e^{2int}}\right\vert $. On the
second half of the interval $\left(  \left(  2k-2\right)  \pi/n,2k\pi
/n\right)  ,$ $\alpha_{\pi/2}^{\prime}\left(  t\right)  =0$ and $\alpha
_{\pi/2}\left(  t\right)  $ is constant.
\end{theorem}

\begin{proof}
Note that the summands $\frac{d}{dt}e^{i\beta/2}h_{n}\left(  e^{it}\right)  $
and $\frac{d}{dt}e^{-i\beta/2}g_{n}\left(  e^{-it}\right)  $ of $\alpha
_{\beta}^{\prime}\left(  t\right)  =\frac{d}{dt}\arg f_{\beta}\left(
e^{it}\right)  $ have the same magnitude, namely $1/\left\vert \sqrt
{1-e^{2int}}\right\vert $. \ From equation (\ref{hgdiff}), the angle between
$\frac{d}{dt}h_{n}\left(  e^{it}\right)  $ and $\frac{d}{dt}\overline{g_{n}%
}\left(  e^{it}\right)  $ is the constant $\left(  -1\right)  ^{j}\pi/2$ on
each interval $\left(  \left(  j-1\right)  \pi/n,j\pi/n\right)  ,$ and the
presence of $\beta$ changes this difference to $\beta+\left(  -1\right)
^{j}\pi/2.$ Thus from the cosine rule (see Remark \ref{cosine}) we obtain the
magnitude $\left\vert \alpha_{\beta}^{\prime}\left(  t\right)  \right\vert
=\frac{\sqrt{2}\sqrt{1+\cos\left(  \beta+\left(  -1\right)  ^{j}\pi/2\right)
}}{\left\vert \sqrt{1-e^{i2nt}}\right\vert }=\frac{\sqrt{2}\sqrt{1+\left(
-1\right)  ^{j}\sin\left(  \beta\right)  }}{\left\vert \sqrt{1-e^{i2nt}%
}\right\vert },$ for $t\neq j\pi/n.$ This proves equation (\ref{mag}). Note
that since $\left\vert \beta\right\vert <\pi/2$, $\alpha_{\beta}^{\prime
}\left(  t\right)  \neq0.$ We turn to the argument of $\alpha_{\beta}^{\prime
}.$ We can utilize arithmetic means involving (\ref{argghh}) and
(\ref{argghg}), \ or simply make use of (\ref{HSarg}). For either approach,
the initial argument $\arg\alpha_{\beta}^{\prime}\left(  0\right)  ^{+}$ must
be determined as either $\pi$ or $0.$ The initial values of $\arg\frac{d}%
{dt}h_{n}\left(  e^{it}\right)  $ and $\arg\frac{d}{dt}\overline{g_{n}}\left(
e^{it}\right)  $ as $t\searrow0^{+}$ are $3\pi/4$ and $-3\pi/4$
respectively$.$ Therefore the initial angle $\arg\left(  \alpha_{0}^{\prime
}\left(  0\right)  ^{+}\right)  $ is $\pi$ rather than $0,$ and $\arg
\alpha_{0}^{\prime}\left(  t\right)  =\pi-\left(  n/2-1\right)  t.$ This
formula for $\arg\alpha_{0}^{\prime}\left(  t\right)  $ holds throughout
$\left(  0,\pi/n\right)  ,$ since $\alpha_{\beta}\left(  t\right)  $ is
continuous$.$ The initial angles $\frac{d}{dt}h_{n}\left(  e^{it}\right)  $
and $\frac{d}{dt}\overline{g_{n}}\left(  e^{it}\right)  $ as $t\searrow0^{+}$
on $\left(  \pi/n,2\pi/n\right)  $ become $3\pi/4+\pi/n$ and $\pi/4+\pi/n$
respectively (using Proposition \ref{linearangles}, or rotational symmetry).
The initial value of $\arg\alpha_{0}^{\prime}\left(  t\right)  $ on $\left(
\pi/n,2\pi/n\right)  $ is therefore $\pi/2+\pi/n,$ consistent with
(\ref{compass}). Thus the formula $\arg\alpha_{0}^{\prime}\left(  t\right)
=\pi-\left(  n/2-1\right)  t$ holds throughout $\left(  0,2\pi/n\right)  $,
where defined. For \linebreak$\left\vert \beta\right\vert <\pi/2,$ the means
of $\arg\frac{d}{dt}e^{i\beta/2}h_{n}\left(  e^{it}\right)  $ and $\arg
\frac{d}{dt}e^{-i\beta/2}\overline{g_{n}}\left(  e^{it}\right)  $ are the same
as for $\beta=0$: the initial angles $\arg\frac{d}{dt}e^{i\beta/2}h_{n}\left(
e^{it}\right)  $ and $\arg\frac{d}{dt}e^{-i\beta/2}\overline{g_{n}}\left(
e^{it}\right)  $ as $t\searrow0^{+}$ on $\left(  0,\pi/n\right)  $ are
respectively $3\pi/4+\beta/2$ (third quadrant) and $-3\pi/4-\beta/2$ (fourth
quadrant)$.$ Thus the initial angle $\arg\left(  \alpha_{\beta}^{\prime
}\left(  0\right)  ^{+}\right)  $ maintains the value $\pi$ (rather than
$0$)$.$ Similarly the initial angles $\arg\frac{d}{dt}e^{i\beta/2}h_{n}\left(
e^{it}\right)  $ and $\arg\frac{d}{dt}e^{-i\beta/2}\overline{g_{n}}\left(
e^{it}\right)  /2$ as $t\searrow0^{+}$ on $\left(  \pi/n,2\pi/n\right)  $ are
respectively $3\pi/4+\pi/n+\beta/2$ and $\pi/4+\pi/n-\beta/2.$ Thus the
initial angle $\arg\left(  \alpha_{\beta}^{\prime}\left(  \pi/n\right)
^{+}\right)  $ also remains fixed as $\pi/2+\pi/n.$ We conclude that the
equation for $\arg\alpha_{0}^{\prime}\left(  t\right)  $ is valid for
$\arg\alpha_{\beta}^{\prime}\left(  t\right)  $ on $\left(  0,2\pi/n\right)  $
(note this would \textit{not} be the case for $\left\vert \beta\right\vert
\in\left(  \pi/2,\pi\right)  $). Thus equation (\ref{compass}) holds for
$k=1,$ with $\alpha_{\beta}$ in place of $\alpha,$ except at $t=\pi/n$ where
$\arg\alpha_{\beta}^{\prime}\left(  t\right)  $ is not defined$.$ By Theorem
\ref{symmetry} (ii), $f_{\beta}\ $has $n$-fold rotational symmetry needed to
invoke Lemma \ref{cusps} (ii), and in view of Remark \ref{extend}, formula
(\ref{compass}) holds on each interval $\left(  \left(  2k-2\right)
\pi/n,2k\pi/n\right)  ,$ for $t\neq\left(  2k-1\right)  \pi/n.$ This proves (i).

We now consider (ii), with $\beta=\pi/2$. From Proposition \ref{linearangles}
(iii), the angle between $\frac{d}{dt}e^{i\pi/4}h_{n}\left(  e^{it}\right)  $
and $\arg\frac{d}{dt}e^{-i\pi/4}\overline{g_{n}}\left(  e^{it}\right)  $ is
$\beta+\left(  -1\right)  ^{j}\pi/2$ (mod $2\pi$), which is either $0$ or
$\pi.$ For even $j,$ we see that $\frac{d}{dt}e^{i\pi/4}h_{n}\left(
e^{it}\right)  $ and $\frac{d}{dt}e^{-i\pi/4}\overline{g_{n}}\left(
e^{it}\right)  $ cancel, since their arguments differ by $\pi.$ Then
$\alpha_{\beta}^{\prime}\left(  t\right)  $ $=0,$ and $\alpha_{\beta}$ is
constant on $\left(  \left(  j-1\right)  \pi/n,j\pi/n\right)  .$ For odd $j,$
the two summands have the same argument, so $\arg\alpha_{\pi/2}^{\prime
}\left(  t\right)  =\arg\frac{d}{dt}e^{i\pi/4}h_{n}\left(  e^{it}\right)  .$
Adding $\pi/4$ to formula (\ref{argghh}), we obtain equation $\arg\alpha
_{\pi/2}^{\prime}\left(  t\right)  =\pi-\left(  \frac{n}{2}-1\right)
t+\left(  j-1\right)  \pi/2$ on $\left(  \left(  j-1\right)  \pi
/n,j\pi/n\right)  .$ But this subinterval is the \textquotedblleft first half"
of the interval $\left(  \left(  2k-2\right)  \pi/n,2k\pi/n\right)  $ where
$j=2k-1,$ so replacing $j$ in our formula for $\arg\alpha_{\pi/2}^{\prime
}\left(  t\right)  $ we obtain $\pi-\left(  \frac{n}{2}-1\right)  t+\left(
2k-2\right)  \pi/2$, which is (\ref{compass}). Moreover the non-zero magnitude
of $\alpha_{\pi/2}^{\prime}\left(  t\right)  $ is $\left\vert \alpha_{\pi
/2}^{\prime}\left(  t\right)  \right\vert =2/\left\vert \sqrt{1-e^{2int}%
}\right\vert .$
\end{proof}

\begin{corollary}
\label{cuspdir}For $\beta\in\left(  -\pi/2,\pi/2\right)  $, let $\alpha
_{\beta}\left(  t\right)  =f_{\beta}\left(  e^{it}\right)  ,$ and
$k=1,2,...,n. $ Then the singular points\ of $\alpha_{\beta}$ occurring at
multiples of $\pi/n\ $are alternately cusps, and removable nodes. At
$t=\left(  2k-1\right)  \pi/n$, the discontinuity in $\arg\alpha_{\beta
}^{\prime}$ is removable, and $\alpha_{\beta}\left(  \left(  2k-1\right)
\pi/n\right)  $ is a removable node. In traversing from the cusp
$\alpha_{\beta}\left(  \left(  2k-2\right)  \pi/n\right)  $ to the cusp
$\alpha_{\beta}\left(  2k\pi/n\right)  ,$ $\alpha_{\beta}$ has total curvature
$\pi-2\pi/n.$ The total curvature over the first half of the interval, is
equal to the total curvature over the second half of the interval $\left(
\left(  2k-2\right)  \pi/n,2k\pi/n\right)  ,$ and is $\pi/2-\pi/n$.
\end{corollary}

\begin{proof}
As noted in the proof of (i) above, Lemma \ref{cusps} still applies with
(\ref{compass}) holding on $\left(  \left(  2k-2\right)  \pi/n,2k\pi/n\right)
$ except at the center $t=\left(  2k-1\right)  \pi/n$ of the interval$.$ We
therefore use (\ref{compass}) on the punctured interval to evaluate the
limits
\begin{equation}
\arg\alpha_{\beta}^{\prime}\left(  \left(  2k-1\right)  \pi/n\right)
^{-}=\arg\alpha_{\beta}^{\prime}\left(  \left(  2k-1\right)  \pi/n\right)
^{+}=\pi/2+\left(  2k-1\right)  \pi/n, \label{nodedir}%
\end{equation}
so the exterior angle is $0$ and $\alpha_{\beta}\left(  \left(  2k-1\right)
\pi/n\right)  $ is a removable node. We note that the discontinuity in
$\arg\alpha_{\beta}^{\prime}$ at $\left(  2k-1\right)  \pi/n$ is removable.
Again by Lemma \ref{cusps}, $\alpha_{\beta}\left(  2k\pi/n\right)  $ is a cusp
and the axis has argument $2k\pi/n.$ The equation (\ref{compass}) is monotonic
in $t$, so the total change in $\arg\alpha_{\beta}^{\prime}$ is equal to
$\pi-2\pi/n$ on the interval $\left(  \left(  2k-2\right)  \pi/n,2k\pi
/n\right)  .$ Since (\ref{compass}) is linear, half of this change, namely
$\pi/2-\pi/n,$ occurs on each half of the interval $\left(  \left(
2k-2\right)  \pi/n,2k\pi/n\right)  .$
\end{proof}

%

\begin{figure}[h]%
\centering
\includegraphics[
height=2.226in,
width=4.3742in
]%
{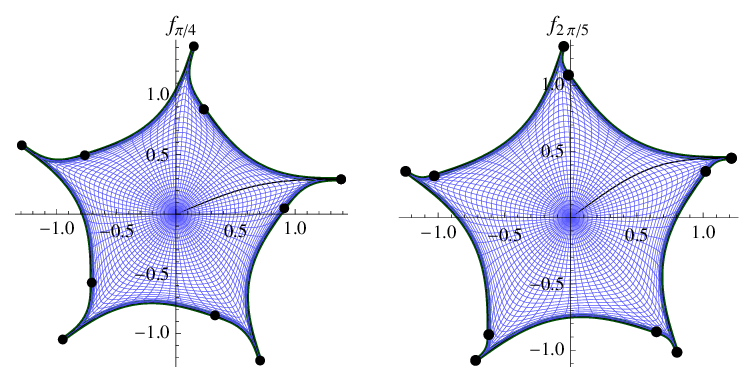}%
\caption{For $n=5$ with $\beta=\pi/4$ (left) and $\beta=2\pi/5$ (right), the
nodes and cusps are indicated by a dots, and interlace with one another by
argument. The argument of each node and cusp increases with $\beta$ by
equation (\ref{cnarg}). Node and cusp locations are as described $\left\vert
\beta\right\vert <\pi/2$ in Theorem \ref{cuspsnodes}. }%
\label{cuspnodepic5}%
\end{figure}
By Corollary \ref{cuspdir}, the rosette harmonic mappings $f_{\beta}$ for
$\left\vert \beta\right\vert <\pi/2$ have $n$-cusps, just as for the
$n$-cusped hypocycloid mappings. Moreover with $n$ fixed, corresponding cusps
for different mappings have cusp axes that are parallel. This can be seen in
Figure \ref{cuspnodepic5} and in Figure \ref{five} where cusps $\alpha_{\beta
}\left(  0\right)  $ have axes parallel to the real axis$.$ The parallelism of
axes follows from the identical unit tangent values of the boundary
extensions, which also explains the total curvature of $\pi-2\pi/n$ from one
cusp to the next, described both in Proposition \ref{hypocycloid} and
Corollary \ref{cuspdir}.

The last graph $f_{\pi/2}\left(  U\right)  $ in Figure \ref{five} does not
have cusps, but nodes with an acute interior angle, which we now examine.

\begin{corollary}
\label{nodes}Let $\alpha_{\pi/2}\left(  t\right)  =f_{\pi/2}\left(
e^{it}\right)  ,$ and let $k=1,2,...,n.$ On the first half of the interval
$\left(  \left(  2k-2\right)  \pi/n,2k\pi/n\right)  ,\ $the total curvature of
$\alpha_{\pi/2}$ is $\pi/2-\pi/n,$ while $\alpha_{\pi/2}$ is constant
otherwise$.$ There is a piecewise smooth parametrization $\tilde{\alpha}%
_{\pi/2},$ with the same graph as $\alpha_{\pi/2}$ over $\partial U$, and just
$n$ singular points at $\tilde{\alpha}_{\pi/2}\left(  2k\pi/n\right)
=\alpha_{\pi/2}\left(  2k\pi/n\right)  $ which are nodes of $\tilde{\alpha
}_{\pi/2}$ with interior angle $\pi/2-\pi/n.$
\end{corollary}

\begin{proof}
Since $\alpha_{\pi/2}^{\prime}\left(  t\right)  =0$ on $\left(  \left(
2k-1\right)  \pi/n,2k\pi/n\right)  ,$ the singularities are not isolated, and
moreover these intervals are arcs of constancy for the boundary function of
$f_{\pi/2}$. We define $\tilde{\alpha}_{\pi/2},$ defined piecewise on
$[0,2\pi)$ by%
\begin{equation}
\tilde{\alpha}_{\pi/2}\left(  t\right)  =\alpha_{\pi/2}\left(  \left(
k-1\right)  \pi/n+t/2\right)  ,\text{ }t\in\lbrack\left(  2k-2\right)
\pi/n,2k\pi/n). \label{halfspeed}%
\end{equation}

Then on each interval in (\ref{halfspeed}) the curve $\tilde{\alpha}_{\pi/2} $
has the values\newline$\alpha_{\pi/2}\left(  \left[  \left(  2k-2\right)
\pi/n,\left(  2k-1\right)  \pi/n\right]  \right)  $. Moreover, $\tilde{\alpha
}_{\pi/2}\left(  2k\pi/n\right)  =\alpha_{\pi/2}\left(  2k\pi/n\right)  .$
Thus $\tilde{\alpha}_{\pi/2}$ is also continuous, with the same image as
$\alpha_{\pi/2}$ on $\left[  0,2\pi\right]  .$ We compute
\begin{align*}
\lim_{t\rightarrow2k\pi/n^{-}}\arg\tilde{\alpha}_{\pi/2}^{\prime}\left(
t\right)   &  =\lim_{t\rightarrow2k\pi/n^{-}}\arg\alpha_{\pi/2}^{\prime
}\left(  \left(  k-1\right)  \pi/n+t/2\right) \\
&  =\lim_{t^{\prime}\rightarrow\left(  2k-1\right)  \pi/n}2k\pi/2-\left(
\frac{n}{2}-1\right)  t^{\prime}=k\pi-\left(  \frac{n}{2}-1\right)  \left(
2k-1\right)  \frac{\pi}{n}.
\end{align*}
and using the formula for $\arg\tilde{\alpha}_{\pi/2}^{\prime}$ on $\left(
2k\pi/n,2\left(  k+1\right)  \pi/n\right)  $ obtain
\begin{align*}
\lim_{t\rightarrow2k\pi/n^{+}}\arg\tilde{\alpha}_{\pi/2}^{\prime}\left(
t\right)   &  =\lim_{t\rightarrow2k\pi/n^{+}}\arg\alpha_{\pi/2}^{\prime
}\left(  k\pi/n+t/2\right) \\
&  =\lim_{t^{\prime}\rightarrow2k\pi/n}\left(  2k+2\right)  \pi/2-\left(
\frac{n}{2}-1\right)  t^{\prime}\\
&  =\left(  k+1\right)  \pi-\left(  \frac{n}{2}-1\right)  2k\frac{\pi}{n}.
\end{align*}
The difference $\arg\tilde{\alpha}_{\pi/2}^{\prime}\left(  2k\pi/n\right)
^{+}-\arg\tilde{\alpha}_{\pi/2}^{\prime}\left(  2k\pi/n\right)  ^{-}$ is thus
$\pi-\left(  \frac{n}{2}-1\right)  \pi/n=\pi/2+\pi/n,$ so we have a node with
interior angle $\pi/2-\pi/n.$
\end{proof}

\begin{remark}
\label{node}The node $\tilde{\alpha}_{\pi/2}\left(  2k\pi/n\right)  $ can be
written $e^{i2k\pi/n}\alpha_{\pi/2}\left(  t\right)  $ for any \newline%
$t\in\left[  \left(  2k-1\right)  \pi/n,2k\pi/n\right]  ,$ where $\alpha
_{\pi/2}$ is constant$.$ We write the nodes of $\tilde{\alpha}_{\pi/2}$ as
\newline$e^{i2k\pi/n}f_{\pi/2}\left(  1\right)  $ when convenient$,$ and refer
to nodes of $\tilde{\alpha}_{\pi/2}$ as the nodes of $f_{\pi/2}.$
\end{remark}

\begin{example}
When $\beta=\pi/2$ and with $n=5,$ then on intervals $\left(  0,\pi/5\right)
,$ $\left(  2\pi/5,3\pi/5\right)  ,$ ... the boundary arcs $e^{i\pi/4}%
h_{5}\left(  e^{it}\right)  $ and $e^{-i\pi/4}\overline{g_{5}}\left(
e^{-it}\right)  $ are translates of one another, and on intervals $\left(
\pi/5,2\pi/5\right)  ,$ $\left(  3\pi/5,5\pi/5\right)  ,$ ..., the boundary
arcs are mirrors of one another. Figure \ref{transmirror} (right) shows the
arc of constancy $\left(  \pi/5,2\pi/5\right)  $ on which $\ e^{i\pi/4}%
h_{5}\left(  e^{it}\right)  $ and $e^{-i\pi/4}\overline{g_{5}}\left(
e^{-it}\right)  $ are mirror images, and where $f_{\pi/2}\left(
e^{it}\right)  $ is equal to the node $e^{i2\pi/5}f_{\pi/2}\left(  1\right)  $
(indicated with a larger dot in the second quadrant).
\begin{figure}[ptb]%
\centering
\includegraphics[
height=2.4059in,
width=5.0548in
]%
{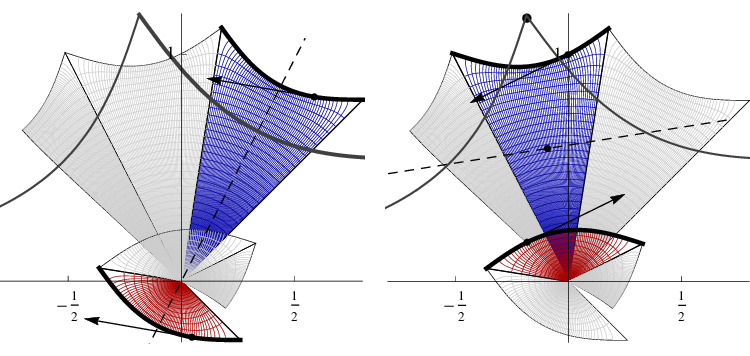}%
\caption{Images of sectors in $U$ of $e^{i\pi/4}h_{5}$ and $e^{-i\pi
/4}\overline{g_{5}}$ with $\arg z\in\left(  0,\pi/5\right)  $ shaded (left)
and $\arg z\in\left(  \pi/5,2\pi/5\right)  $ shaded (right), for $n=5.$ The
bounding curves (thickened)$,$ are translates (left) and reflections (right).
A tangent is indicated for $h_{5}\left(  e^{it}\right)  $ and $\overline
{g_{5}\left(  e^{it}\right)  }$ in each case, along with a portion of
$f_{\pi/2}\left(  e^{it}\right)  $, which is constant (right) on $\left(
\pi/5,2\pi/5\right)  .$}%
\label{transmirror}%
\end{figure}

\end{example}

We now complete our description of the symmetries within the graphs of the
rosette mappings.

\begin{corollary}
\label{noreflect}Let $\beta\in%
\mathbb{R}
$ and $n\geq3.$ If $\beta$ is not a multiple of $\pi/2,$ then the image set
$f_{\beta}\left(  U\right)  $ does not have reflectional symmetry. In this
case, $f_{\beta}\left(  U\right)  $ has symmetry group $C_{n}.$ Otherwise,
$\beta$ is a multiple of $\pi/2,$ and $f_{\beta}\left(  U\right)  $ has
symmetry group $D_{n}.$
\end{corollary}

\begin{proof}
All rosette mappings have at least $n$ fold rotational symmetry, by
Proposition \ref{symmetry} (ii). Since $f_{\beta}$ has either exactly $n$
cusps, or exactly $n$ non-removable nodes, $f_{\beta}\left(  U\right)  $
cannot have a higher order of symmetry. For $\left\vert \beta\right\vert
\in\left(  0,\pi/2\right)  ,$ the axis of the cusp through $f_{\beta}\left(
1\right)  $ is parallel to the real axis, while the radial ray through $0$ and
$f_{\beta}\left(  1\right)  $ has argument $\psi=\arg f_{\beta}\left(
1\right)  ,$ distinct for each $\beta$ in $\left(  -\pi/2,\pi/2\right)  $ by
equation (\ref{cnarg})$.$ Moreover as noted in Lemma \ref{convexconcave},
$\psi$ is acute$,$ and has the same sign as $\beta.$ Thus if $\left\vert
\beta\right\vert \in\left(  0,\pi/2\right)  ,$ then any reflection of
$f_{\beta}\left(  U\right)  $ has angle of opposite signresults changing the
sign of the angle between the reflected cusp axis and the reflected radial
ray, resulting in a distinct reflected image set. Thus the symmetry group of
$f_{\beta}\left(  U\right)  $ is $C_{n}$ for $\left\vert \beta\right\vert
\in\left(  0,\pi/2\right)  .$ This fact extends by formula (\ref{trans}) to
any real $\beta$ that is not a multiple of $\pi/2.$ If $\beta=0$ or $\beta
=\pi/2$ we already established that $f_{\beta}\left(  U\right)  $ has
reflectional symmetry. We conclude that the sets $f_{0}\left(  U\right)  $ and
$f_{\pi/2}\left(  U\right)  $ have dihedral symmetry group $D_{n}$. If
$\beta=l\pi/2$ for some $l\in%
\mathbb{Z}
$ then by formula (\ref{trans}), $f_{\beta}\left(  U\right)  $ is a rotation
of either $f_{\pi/2}\left(  U\right)  $ or $f_{0}\left(  U\right)  ,$ and thus
has reflectional symmetry also.
\end{proof}

\begin{corollary}
\label{differ}Let $n\geq3.$ For distinct $\beta$ and $\beta^{\prime}$ in the
interval $(-\pi/2,\pi/2],$ the image sets $f_{\beta}\left(  U\right)  $ and
$f_{\beta^{\prime}}\left(  U\right)  $ are not scalings or rotations of one
another. Moreover with the parameter $\beta$ within the set $\left[
0,\pi/2\right]  ,$ all images of the unit disk under a rosette harmonic
mapping are obtained, up to rotation and reflection.
\end{corollary}

\begin{proof}
The proof of Corollary \ref{noreflect} shows that if $\beta\in\left(
-\pi/2,\pi/2\right)  ,$ the angle between the cusp axis through $f_{\beta
}\left(  1\right)  $ and the radial line from $0$ to $f_{\beta}\left(
1\right)  $ intersect at an angle that is distinct for each choice of $\beta.
$ Thus $f_{\beta}\left(  U\right)  $ is different from any rotation or scaling
of $f_{\beta^{\prime}}\left(  U\right)  $ for any other $\beta^{\prime}%
\in\left(  -\pi/2,\pi/2\right)  .$ Moreover $f_{\pi/2}$ is the only $f_{\beta
}$ without cusps for $\beta\in(-\pi/2,\pi/2].$ To prove the second statement,
let $\tilde{\beta}\in%
\mathbb{R}
\backslash(-\pi/2,\pi/2].$ Upon reducing $\tilde{\beta}$ modulo $\pi,$ we
obtain an equivalent $\beta\equiv\tilde{\beta}$ $(\operatorname{mod}$ $\pi)$
where $\beta\in(-\pi/2,\pi/2].$ By possibly repeated application of
(\ref{rotbet}) in Proposition \ref{parametrize} (i), $f_{\tilde{\beta}}\left(
z\right)  $ is obtained as a pre- and post- composition of a rotation with
$f_{\beta}\left(  z\right)  .$ Thus it is sufficient to consider $\beta
\in(-\pi/2,\pi/2].$ If $\beta<0,$ then by Proposition \ref{parametrize} (ii),
$f_{\beta}\left(  z\right)  $ is a reflection in the real axis of $f_{-\beta
}\left(  z\right)  $ where $-\beta\in\left[  0,\pi/2\right]  .$
\end{proof}

We turn to describing features of the graphs $f_{\beta}\left(  U\right)  $ for
$\beta$ restricted to $(-\pi/2,\pi/2].$ We may use equation (\ref{trans}) to
identify cusps and nodes for $\beta$ outside this interval, but we make the
cautionary observation that while for $\beta\in\left(  -\pi/2,\pi/2\right)  $
$f_{\beta}\left(  1\right)  $ is a cusp and $f_{\beta}\left(  e^{i\pi
/n}\right)  $ is a node, adding $\pi$ to $\beta$ results in $f_{\beta+\pi
}\left(  1\right)  $ being a node and $f_{\beta+\pi}\left(  e^{i\pi/n}\right)
$ being a cusp. Quantities such as the magnitude of a cusp, the magnitude of a
node, and the distribution of angle between cusps neighboring nodes however$,$
are independent of the interval to which $\beta$ is restricted$.$ These
quantities are derived in Theorem \ref{cuspsnodes}, with the help of Lemma
\ref{convexconcave}.

Theorem \ref{cuspsnodes} shows that for fixed $n,$ the image $f_{0}\left(
U\right)  $ has maximal diameter, and $diam\left(  f_{\left\vert
\beta\right\vert }\left(  U\right)  \right)  $ decreases with $\left\vert
\beta\right\vert \in\left(  0,\pi/2\right)  .\,$This is illustrated in Figure
\ref{five}, where all four graphs are plotted on the same scale. We also see
that for $\beta=0,$ the cusps and nodes are equally spaced with angular
separation $\pi/n,$ but this becomes increasingly unbalanced as $\left\vert
\beta\right\vert $ increases to $\pi/2.$ As $\beta\nearrow\pi/2^{-},$ each
node approaches the subsequent\footnote{Subsequent cusp (or node) here
indicates the cusp (or node) with the next largest argument.} cusp, as shown
in Theorem \ref{cuspsnodes} (ii).

\begin{theorem}
\label{cuspsnodes}Let $n\geq3,$ $\beta\in(-\pi/2,\pi/2],$ and let
$\alpha_{\beta}\left(  t\right)  =f_{\beta}\left(  e^{it}\right)  $ where
$f_{\beta}$ is a rosette harmonic mapping$.$ Recall the constant $K_{n}$
defined in Lemma \ref{convexconcave}$.$ \newline(i)$\ $If $\beta\in\left(
-\pi/2,\pi/2\right)  $ then the magnitude of the cusps decreases strictly if
$\left\vert \beta\right\vert \in\lbrack0,\pi/2),$ from a maximum $K_{n}%
\sec\left(  \frac{\pi}{2n}\right)  ,$ and with infimum $K_{n}\left(
1+\tan\left(  \frac{\pi}{2n}\right)  \right)  .$ The removable nodes have
magnitudes that increase strictly if $\left\vert \beta\right\vert \in
\lbrack0,\pi/2),$ from the minimum $K_{n}\left(  1+\tan\left(  \frac{\pi}%
{2n}\right)  \right)  $ and with supremum $K_{n}\sec\left(  \frac{\pi}%
{2n}\right)  .$\newline(ii) If $\beta\in\left(  -\pi/2,\pi/2\right)  $ then
the angle between a cusp neighboring node is
\begin{equation}
\pi/n\pm\arctan\left(  \frac{2\tan\left(  \frac{\pi}{2n}\right)  \sin\left(
\beta\right)  }{1-\tan^{2}\left(  \frac{\pi}{2n}\right)  }\right)  ,
\label{sep}%
\end{equation}
where we use the positive sign when the node has the more positive argument.
For $\beta=0,$ the cusps and neighboring nodes are separated by equal angles
of $\pi/n$. The difference in arguments of a cusp and subsequent node
increases from $0$ to $2\pi/n$ with $\beta\in\left(  -\pi/2,\pi/2\right)  $.
\newline(iii) For $\beta=\pi/2,$ the $n$ nodes of the rosette mapping
$f_{\pi/2}$ have magnitude \linebreak$K_{n}\sec\frac{\pi}{2n},$ and for the
node $f_{\pi/2}\left(  1\right)  $ we have $\arg\left(  f_{\pi/2}\left(
1\right)  \right)  =\pi/4-\pi/\left(  2n\right)  $.
\end{theorem}

\begin{proof}
We begin with (i). Due to the $\cos\beta$ term in (\ref{cnmag}), the cusp
$\left\vert f_{\beta}\left(  1\right)  \right\vert $ is decreasing with
$\left\vert \beta\right\vert $ and the node $\left\vert f_{\beta}\left(
e^{i\pi/n}\right)  \right\vert $ is increasing with $\left\vert \beta
\right\vert .$ Note that
\[
\sec^{2}\left(  \frac{\pi}{2n}\right)  \pm2\tan\left(  \frac{\pi}{2n}\right)
\cos\beta=1+\tan^{2}\left(  \frac{\pi}{2n}\right)  \pm2\tan\left(  \frac{\pi
}{2n}\right)  =\left(  1\pm\tan\left(  \frac{\pi}{2n}\right)  \right)  ^{2}.
\]
In view of this equation, the maximum of $\left\vert f_{\beta}\left(
1\right)  \right\vert $ is $K_{n}\left(  1+\tan\left(  \frac{\pi}{2n}\right)
\right)  $ when $\beta=0,$ and $\left\vert f_{\beta}\left(  1\right)
\right\vert $ approaches $K_{n}\sec\left(  \frac{\pi}{2n}\right)  $ as $\beta$
approaches $\pi/2.$ Similarly, the minimum of $\left\vert f_{\beta}\left(
e^{i\pi/n}\right)  \right\vert $ is $K_{n}\left(  1-\tan\left(  \frac{\pi}%
{2n}\right)  \right)  $ when $\beta=0,$ and approaches $K_{n}\left(
\sec\left(  \frac{\pi}{2n}\right)  \right)  $ as an upper bound as $\left\vert
\beta\right\vert $ approaches $\pi/2,$ proving (i)$.$ For (ii) formula
(\ref{cnarg}) shows us $\arg f_{\beta}\left(  1\right)  $ and $\arg f_{\beta
}\left(  e^{i\pi/n}\right)  $ are increasing with $\beta$ on $\left(
-\pi/2,\pi/2\right)  .$ Clearly $\arg\left(  f_{0}\left(  1\right)  \right)
=0$ and $\arg f_{0}\left(  e^{i\pi/n}\right)  =\pi/n.$ Additionally we have
$\arg\left(  f_{\beta}\left(  1\right)  \right)  <\pi/n<\arg f_{\beta}\left(
e^{i\pi/n}\right)  \,\ $for $\beta\in\left(  0,\pi/2\right)  .$ We compute the
angle between this cusp and node to be $\arg f_{\beta}\left(  e^{i\pi
/n}\right)  -\arg f_{\beta}\left(  1\right)  =\pi/n+\psi^{\prime}-\psi,$ and
we use the arctangent formula $\arctan\psi^{\prime}-\arctan\psi=\arctan\left(
\frac{\psi^{\prime}-\psi}{1+\psi^{\prime}\psi}\right)  $ where $\psi^{\prime}$
and $\psi$ are as defined in equation (\ref{cnarg}). Thus after a short
calculation we obtain formula (\ref{sep}). This difference increases from
$\pi/n$ and approaches $2\pi/n$ as $\beta$ increases through the interval
$[0,\pi/2)$. Now suppose that $\beta\in\left(  -\pi/2,0\right)  .$ The
reflection of the node $f_{\beta}\left(  e^{i\pi/n}\right)  $ in the real axis
is the node $f_{-\beta}\left(  e^{-i\pi/n}\right)  $ of $f_{-\beta}$ and the
reflection of $f_{\beta}\left(  1\right)  $ is the cusp $f_{-\beta}\left(
1\right)  .$ From rotational symmetry, the angular difference between $\arg
f_{-\beta}\left(  1\right)  $ and $\arg f_{-\beta}\left(  e^{-i\pi/n}\right)
$ is the same as the that of $\arg f_{-\beta}\left(  e^{i2\pi/n}\right)  $ and
$\arg f_{-\beta}\left(  e^{i\pi/n}\right)  .$ Again using rotational symmetry,
this is $2\pi/n-\left(  \psi-\psi^{\prime}\right)  =\pi/n-\arctan\left(
\frac{\psi^{\prime}-\psi}{1+\psi^{\prime}\psi}\right)  =\pi/n-\arctan\left(
\frac{2\tan\left(  \frac{\pi}{2n}\right)  \sin\left(  -\beta\right)  }%
{1-\tan^{2}\left(  \frac{\pi}{2n}\right)  }\right)  =\pi/n+\arctan\left(
\frac{2\tan\left(  \frac{\pi}{2n}\right)  \sin\left(  \beta\right)  }%
{1-\tan^{2}\left(  \frac{\pi}{2n}\right)  }\right)  $. For (iii), with
$\beta=\pi/2, $ we noted in Remark \ref{node} after Corollary \ref{nodes} that
the nodes can be expressed as rotations of $f_{\pi/2}\left(  1\right)  ,$ for
which we observed the stated quantities in Lemma \ref{convexconcave} (i)$.$
\end{proof}

The next example illustrates the separation of cusps and nodes as it varies
with $\beta,$ as described in Theorem \ref{cuspsnodes} (see also Figure
\ref{cuspnodepic5})$.$ The relative proximity of a node and neighboring cusp
coupled with equal total curvature of the boundary curve between any node and
cusp (Corollary \ref{cuspdir}) gives rise to the appearance of a cresting wave
at the cusp.

\begin{example}
\label{cuspnody}For $n=5,$ the arguments of the nodes and cusps of $f_{0}$ are
equally spaced by $\pi/n=\pi/5.$ By Lemma \ref{convexconcave}, as $\beta
\in\lbrack0,\pi/2)$ increases, the separation of a cusp and subsequent node
grows from $\pi/5=36^{\circ}$ towards $2\pi/5=72^{\circ}. $ For $\beta=\pi/4$
this separation is $\pi/5+\arctan\sqrt{5/2-\sqrt{5}}\approx63^{\circ},$ while
for $\beta=2\pi/5$ it grows to $\pi/5+\arctan\sqrt{5/4-\sqrt{5}/4}%
\approx71^{\circ}$ (see Figure \ref{cuspnodepic5}). For a specific cusp or
node $f_{\beta}\left(  e^{ij_{0}\pi/n}\right)  ,$ $\arg f_{\beta}\left(
e^{ij_{0}\pi/n}\right)  $ increases with $\beta$ (see proof of Theorem
\ref{cuspsnodes} (ii))$,\ $which is also illustrated in Figure
\ref{cuspnodepic5} as $\beta$ increases from $\pi/4$ to $2\pi/5.$ By Corollary
\ref{cuspdir}, the total curvature of the boundary of $f_{\beta}$ between
neighboring cusps is $\pi-2\pi/n=3\pi/5=108^{\circ}.$ The total curvature of
the boundary from a cusp to a neighboring node is half of this, namely
$54^{\circ}$. Figure \ref{cuspnodepic5} also indicates the phenomenon
described in Theorem \ref{cuspsnodes}, that while both nodes and cusps
``rotate counterclockwise" as $\beta\in\left(  0,\pi/2\right)  $ increases,
the nodes do so at a greater rate..
\end{example}

Finally we point out that the argument of the boundary curve fails to be
strictly increasing on the whole of $\partial U,$ for $\beta$ with $\left\vert
\beta\right\vert \in\left(  0,\pi/2\right)  .$

\begin{corollary}
\label{parallel}Let $n\geq3$, $\beta\in\left(  -\pi/2,\pi/2\right)  $ and
$\alpha_{\beta}$ be the boundary function of the rosette harmonic mapping
$f_{\beta}.$ Then for each $k=1,2,...,n,$ $\arg\alpha_{\beta}\left(  \left(
2k-2\right)  \pi/n\right)  ,\ $\newline$\arg\alpha_{\beta}\left(  \left(
2k-1\right)  \pi/n\right)  ,$ and $\arg\alpha_{\beta}\left(  2k\pi/n\right)  $
occur in increasing order$,$ but there exists an interval on which
$\arg\left(  \alpha_{\beta}\left(  t\right)  \right)  $ is decreasing.
\end{corollary}

\begin{proof}
Theorem \ref{cuspsnodes} (ii) shows that the angle $\gamma$ between\ a cusp
and subsequent node is given by formula (\ref{sep}) which belongs to $\left(
0,2\pi/n\right)  $. Because the cusps are separated by angle $2\pi/n,$ the
angle from a cusp to a subsequent node is then $2\pi/n-\gamma\in\left(
0,2\pi/n\right)  .$ Finally with $\beta\in\left(  0,\pi/2\right)  ,$ suppose
that $\arg\alpha_{\beta}\left(  -\epsilon\right)  \leq\psi=\arg\alpha_{\beta
}\left(  0\right)  $ for some $-\epsilon\in\left(  -\pi/n,0\right)  ,$ but for
which $\arg\alpha_{\beta}^{\prime}\left(  -\epsilon\right)  <\tan\psi.$ Such
an $\epsilon$ exists because $\alpha_{\beta}$ satisfies (\ref{cuspslopes})
with $k=0,$ so $\arg\alpha_{\beta}^{\prime}\left(  0\right)  ^{-}=0.$ Since
$\arg\alpha_{\beta}^{\prime}\left(  t\right)  $ is decreasing, we have
$\arg\alpha_{\beta}\left(  t\right)  <\psi$ for all $t\in\left(
-\epsilon/2,0\right)  ,$ and so $\alpha_{\beta}\left(  t\right)  $ lies in a
half-plane formed by a line parallel to to $\arg z=\psi,$ but with smaller
imaginary part, and this half plane does not contain $\alpha_{\beta}\left(
0\right)  .$ This contradicts continuity of $\alpha_{\beta}\left(  t\right)
\rightarrow\alpha_{\beta}\left(  0\right)  $ as $t\nearrow0^{-}.$ Thus there
is an interval $\left(  -\epsilon,0\right)  $ on which $\arg\alpha_{\beta
}\left(  t\right)  >\psi.$ On this interval, $\arg\alpha_{\beta}\left(
t\right)  $ decreases to $\psi,$ and so rotates negatively relative to the
origin. A similar argument holds when $\beta\in\left(  -\pi/2,0\right)  .$
\end{proof}

\section{Univalence and Fundamental Sets}

Our approach to proving the univalence of $f_{\beta}$ is to use the argument
principle for harmonic functions. We note that various proofs of univalence
for rosette mappings $f_{\beta}$ are possible. The following theorem describes
the winding number of the boundary curve $\alpha\left(  t\right)  =f_{\beta
}\left(  e^{it}\right)  ,$ so that we can apply the argument principle.

\begin{lemma}
\label{unibdry}For fixed $\beta\in\lbrack0,\pi/2),$ the boundary
$\alpha_{\beta}\left(  t\right)  =f_{\beta}\left(  e^{it}\right)  ,$
$t\in\partial U$ is a simple, positively oriented closed curve. While
$\alpha_{\pi/2}$ has arcs of constancy, the parametrization $\tilde{\alpha
}_{\pi/2}$ of equation (\ref{halfspeed}) is a simple, positively oriented
closed curve on $\partial U.$
\end{lemma}

\begin{proof}
Let $\beta\in\left(  0,\pi/2\right)  .$ We give a separate argument for
$\beta=0$ and for $\beta=\pi/2.$ We first show that $\alpha_{\beta}$ is one to
one when restricted to $\left(  0,2\pi/n\right)  ,$ and that the boundary
curve portion $\alpha_{\beta}\left(  \left(  0,2\pi/n\right)  \right)  $ lies
in a sector $S.$ We then show that the portion of the graph of $\alpha_{\beta
}$ restricted to the interval $\left(  2k\pi/n,2\left(  k+1\right)
\pi/n\right)  $ is contained within the set $e^{i2k\pi/n}S.$ The sets
$\left\{  e^{i2k\pi/n}S:k=0,1,...,n-1\right\}  $ are then seen to be pairwise
disjoint, so that the curve $\alpha_{\beta}$ has no self intersections on
$[0,2\pi),$ and we conclude $\alpha_{\beta}$ is a simple closed curve on
$\partial U.$

Define $L_{k}$ to be the axis of the cusp $\alpha_{\beta}\left(  2k\pi\right)
,$ so with argument equal to $2k\pi/n,$ $k=0,1,..,n-1.$ Recall from Theorem
(\ref{fbdry}) that $\alpha_{\beta}$ satisfies (\ref{compass}) (which holds
except at $t=\left(  2k+1\right)  \pi/n$). We begin with $\alpha_{\beta}$ on
the interval $\left(  0,2\pi/n\right)  .$ Let $p$ be the intersection of
$L_{0}$ (parallel to the real axis) and $L_{1}.$ These non-collinear lines
form the boundary of four unbounded open connected sectors, and we define $S$
to be the sector with vertex $p$ \ for which the (distinct) cusps
$\alpha_{\beta}\left(  0\right)  \in L_{0}$ and $\alpha_{\beta}\left(
2\pi/n\right)  \in L_{1}$ belong to $\partial S.$ By equation (\ref{compass}%
)$,$ $\arg\alpha_{\beta}^{\prime}\left(  t\right)  $ is decreasing and
$\arg\alpha_{\beta}^{\prime}\left(  0\right)  ^{+}=\pi.$ Thus the curve
$\alpha_{\beta}$ lies \textquotedblleft above" $L_{0}$. If $\alpha_{\beta}$
were to cross $L_{1}$ at some $t_{1}\in\left(  0,2\pi/n\right)  ,$ then
$\arg\alpha_{\beta}^{\prime}\left(  t_{1}^{\prime}\right)  <2\pi/n$ for some
$t_{1}^{\prime}\in\left(  t_{1},2\pi/n\right)  $ in order for $\arg
\alpha_{\beta}^{\prime}\left(  2\pi/n\right)  ^{-}=2\pi/n$ in equation
(\ref{cuspslopes}) to hold. However $\arg\alpha_{\beta}^{\prime}\left(
\left(  0,2\pi/n\right)  \right)  \subset\left(  2\pi/n,\pi\right)  $ so no
such intersection can occur, and we conclude that $\alpha_{\beta}$ restricted
to $\left(  0,2\pi/n\right)  $ lies within the region $S.$ Moreover, because
$\arg\alpha_{\beta}^{\prime}\left(  t\right)  $ decreases strictly, with total
change $\pi-2\pi/n<\pi,$ $\alpha_{\beta}$ must be one to one on $\left(
0,2\pi/n\right)  .$ Now let $k=1,2,..,n,$ and consider the curve
$\alpha_{\beta}$ restricted to $\left(  2k\pi/n,2\left(  k+1\right)
\pi/n\right)  .$ By rotational symmetry of $f_{\beta}$, $\alpha_{\beta}\left(
t+2k\pi/n\right)  =e^{i2k\pi/n}\alpha_{\beta}\left(  t\right)  $,
$\alpha_{\beta}$ is also one to one on $\left(  2k\pi/n,2\left(  k+1\right)
\pi/n\right)  $, and the graph $\alpha_{\beta}\left(  \left(  2k\pi/n,2\left(
k+1\right)  \pi/n\right)  \right)  \subseteq e^{i2k\pi/n}S.$ We complete the
proof for $\beta\in\left(  0,\pi/2\right)  $ by showing that the sets
$\left\{  e^{i2k\pi/n}S:k=0,1,...,n-1\right\}  $ are pairwise disjoint.
Because of rotational symmetry, the line $e^{i2k\pi/n}L_{0}$ passes through
$e^{i2k\pi/n}\alpha_{\beta}\left(  0\right)  =\alpha_{\beta}\left(
2k\pi/n\right)  ,$ a cusp, and $e^{i2k\pi/n}L_{0}$ has argument $2k\pi/n,$
so\linebreak\ $e^{i2k\pi/n}L_{0}=L_{k}.$ Similarly $e^{i2k\pi/n}L_{1}%
=L_{k+1}.$ Thus $e^{i2k\pi/n}S$ is a sector with sides $L_{k}$ and $L_{k+1},$
and vertex at their point of intersection $e^{i2k\pi/n}p.$ The points
$\left\{  e^{i2k\pi/n}p:k=0,..,n-1\right\}  $ form the vertices of a regular
$n$-gon, that we denote by $P.$ Then $P$ is centered on the origin, and we
have $\operatorname{Im}p>0,$ since the cusp $\alpha_{\beta}\left(  0\right)  $
has argument $\psi=\arg\alpha_{\beta}\left(  0\right)  \in\left(
0,\pi/4\right)  $ by Lemma \ref{convexconcave} (ii). Note also that
$L_{n}=L_{0},$ and $L_{n-1}=e^{-i2\pi/n}L_{0}.$ Thus $L_{0}$ and $L_{n-1}$
intersect at $e^{-i2\pi/n}p,$ with $\operatorname{Im}\left(  e^{-i2\pi
/n}p\right)  =\operatorname{Im}\left(  p\right)  .$ Therefore $p$ is in the
second quadrant, and $e^{-i2\pi/n}p$ is in the first quadrant. We conclude
that one bounding side of the sector $S$ is $\left\{  z\in L_{0}:\arg
z\leq\arg p\right\}  $. By rotational symmetry, the second side of $S$
is\linebreak\ $\left\{  z\in L_{1}:\arg z\leq2\pi/n+\arg p\right\}  $. Thus
the rays $\left\{  z\in L_{k}:\arg z\leq2k\pi/n+\arg p\right\}  $ that extend
the sides of $P,$ each originating at $e^{i2k\pi/n}p$ and extending in the
direction of decreasing argument, form the boundaries of the sectors
$e^{i2k\pi/n}S,$ where $k=0,1,...,n-1$ (see Figure 7 (left) where $S$ is shaded).%

{\includegraphics[
height=2.2364in,
width=2.4561in
]%
{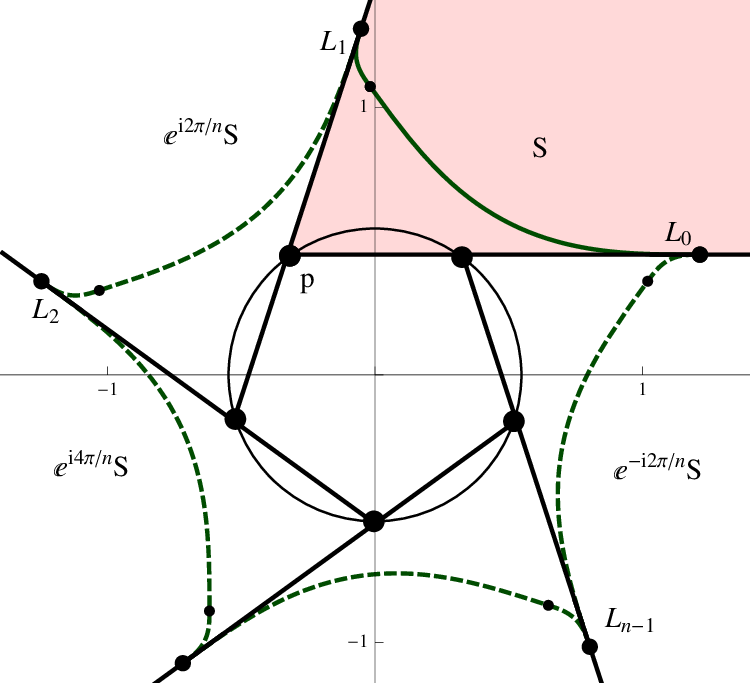}%
}%
\ \
{\includegraphics[
height=2.1724in,
width=2.1136in
]%
{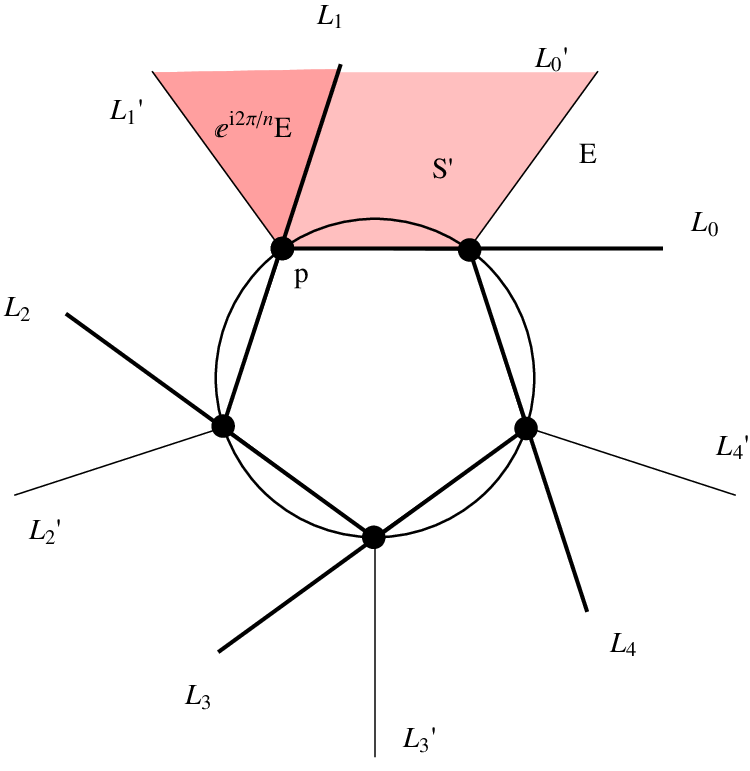}%
}%
\ \ \newline\textsc{Figure 7.} {\small The sector }$S${\small \ and regular
}$n${\small -gon }$P${\small \ (left). The set }$S^{\prime}$ {\small lies on
the non-zero side of the line }$L_{0},${\small \ and is bounded by }%
$L_{0}^{\prime}${\small \ and }$L_{1}^{\prime}${\small \ (right).\medskip}

Various approaches are possible to demonstrate that the sectors $e^{i2k\pi
/n}S$ are disjoint. For instance, note that $\arg p=\pi/2+\pi/n,$ and consider
the related sector
\[
S^{\prime}=\left\{  z\in\text{ext}\left(  P\right)  :\arg z\in\lbrack\pi
/2-\pi/n,\pi/2+\pi/n)\right\}  .
\]
The sets $e^{i2k\pi/n}S^{\prime}$ are clearly disjoint for $k=0,1,...,n-1,$
since the arguments of points in the sets $e^{i2k\pi/n}S^{\prime}$ for
distinct $k\in\left\{  0,1,...,n-1\right\}  $ are in disjoint intervals (see
the right of Figure 7, where $S^{\prime}$ is region that is shaded). Let
$L_{k}^{\prime}$ be the line through the origin and with argument
$\pi/2+\left(  2k-1\right)  \pi/n,$ and let $E$ be the set $E=\left\{  z\in
ext\left(  P\right)  :\arg z<\pi/2-\pi/n,\text{ }\operatorname{Im}%
z\geq\operatorname{Im}p\right\}  ,$ bounded by $L_{0}$ and $L_{0}^{\prime}$
and $P.$ We obtain $S$ from $S^{\prime}$ by including the set $E,$ and
excluding the set $e^{i2k\pi/n}E$ (darker shaded region in the right of Figure
7) from $S^{\prime}.$ Then $S=S^{\prime}\cup E\backslash e^{i2\pi/n}E,$ and by
rotational symmetry,
\[
e^{i2k\pi/n}S=e^{i2k\pi/n}S^{\prime}\cup e^{i2k\pi/n}E\backslash e^{i2\left(
k+1\right)  \pi/n}E.
\]
Thus to obtain $e^{i2k\pi/n}S$ we have simply removed a sector with vertex
$e^{i2\left(  k-1\right)  \pi/n},$ namely $e^{i2k\pi/n}E,$ from the set
$e^{i2k\pi/n}S^{\prime}$ and included its rotation $e^{i2\left(  k+1\right)
\pi/n}E$ to obtain $e^{i2\left(  k+1\right)  \pi/n}S^{\prime}.$ Because the
sets $e^{i2k\pi/n}S^{\prime}$ are disjoint, so are the sets $\left\{
e^{i2k\pi/n}S:k=0,1,..,n-1\right\}  .$

If $\beta=0$ then $p=0$ and the argument above does not apply, but a simple
proof follows if we adapt the proof for $\beta\in\left(  0,\pi/2\right)  $ and
define the sector $S$ to be $S=\left\{  z\in%
\mathbb{C}
:\arg z\in\left(  0,2\pi/n\right)  \right\}  ,$ in which case the sectors
$\left\{  e^{2k\pi i/n}S:k=0,1,...,n-1\right\}  $ are clearly disjoint$.$

If $\beta=\pi/2,$ then we adapt the proof so that $L_{k}$ is the line passing
through node $\alpha_{\pi/2}\left(  2k\pi/n\right)  $ with argument
$\alpha_{\pi/2}^{\prime}\left(  2k\pi/n\right)  ^{+}$. The sectors then are
bounded by the lines $L_{k},$ and the same same argument applies to show
$\tilde{\alpha}_{\pi/2}$ is one to one on $S$ (note that $\arg\tilde{\alpha
}_{\pi/2}\left(  \left(  0,2\pi/n\right)  \right)  \subset\left(  \pi
/2+\pi/\left(  2n\right)  ,\pi\right)  $ rather than $\arg\alpha_{\beta
}^{\prime}\left(  \left(  0,2\pi/n\right)  \right)  \subset\left(  2\pi
/n,\pi\right)  $).
\end{proof}

\begin{theorem}
\label{univalent}For any $n\geq3$ in $%
\mathbb{N}
,$ the harmonic functions $f_{\beta}\left(  z\right)  $, $\beta\in%
\mathbb{R}
$ defined in Definition \ref{defnharmonic} are univalent, and so they are
harmonic mappings.
\end{theorem}

\begin{proof}
We apply the argument principle for harmonic functions of
\cite{DurenHenLaugesen} to obtain univalence of $f_{\beta}.$ We proved in
Lemma \ref{unibdry} that for $\beta\in\lbrack0,\pi/2),$ the boundary curve
$\alpha_{\beta}\left(  t\right)  =f_{\beta}\left(  e^{it}\right)  $ on
$\partial U$ is a simple closed curve about the origin$.$ Although
$\alpha_{\pi/2}$ has arcs of constancy, the winding number about each point
enclosed by the curve $\alpha_{\beta}$ is still $1,$ so the argument principle
applies for $\beta=\pi/2$. Choose an arbitrary point $w_{0}$ enclosed by the
curve $\alpha_{\beta}.$ Then define $f=f_{\beta}-w_{0},$ a harmonic function
continuous in $\bar{U}$, which does not have a zero on $\partial U$. Moreover,
since $\left\vert \omega_{f}\right\vert =\left\vert \omega\right\vert <1$ in
$U$, $f$ does not have any singular zeros in $U$. We see that $f\left(
e^{it}\right)  =\alpha_{\beta}\left(  t\right)  -w_{0}$ has index 1 about the
origin for $t\in\partial U$, so it follows from the argument principle that
$f$ has exactly one zero in $U$. Thus $f_{\beta}\left(  z_{0}\right)  =w_{0}$
for a unique $z_{0}\in U.$ Since the choice of $z_{0}\in U$ was arbitrary, we
see that $f_{\beta}$ is onto the region enclosed by $\alpha_{\beta}.$ If
$w_{1}$ is in the exterior of the region enclosed by the curve $\alpha_{\beta
},$ then consider the function $\tilde{f}=$ $f_{\beta}-w_{1}.$ The harmonic
function $\tilde{f}$ satisfies the same hypotheses as did $f,$ but the winding
number of $\tilde{f}\left(  e^{it}\right)  $ about the origin is zero. Thus
there is no $z_{1}\in U$ for which $f_{\beta}\left(  z_{1}\right)  =w_{1}.$
Therefore $f_{\beta}\left(  U\right)  $ is contained in the interior of the
region bounded by $\alpha_{\beta}.$ If $\beta\in(-\pi/2,0),$ then $f_{\beta}$
is univalent since it is a reflection of $f_{-\beta}$ where $-\beta>0.$
Finally if $\beta\not \in (-\pi/2,\pi/2]$ then $f_{\beta}\left(  z\right)  $
is a rotation of $f_{\beta}\left(  e^{-il\pi/n}z\right)  $ for some $l\in%
\mathbb{Z}
$ by Corollary \ref{transit}, so is also univalent.
\end{proof}

\begin{remark}
On sectors $S$ of the closed unit disk for which $g_{n}\left(  S\right)  $ is
convex, we can show that $g_{n}$ is relatively more contractive than $h_{n},$
in that for any two points $z_{0}$ and $z_{1}$ in the sector,
\[
\left\vert h_{n}\left(  z_{0}\right)  -h_{n}\left(  z_{1}\right)  \right\vert
\geq\left\vert g_{n}\left(  z_{0}\right)  -g_{n}\left(  z_{1}\right)
\right\vert
\]
Moreover the inequality is strict when at least one point is not on $\partial
U.$ This fact can be used to show that $f_{\beta}$ is one to one in the
sectors $S$ of $\bar{U}$ with argument in the range $[0,\pi/n)$, or with
argument in the range $[\pi/n,2\pi/n).$ This leads to a direct proof of
univalence that does not rely on the argument principle.
\end{remark}

We now define a fundamental set, rotations of which make up the graph of the
rosette harmonic mapping $f_{\beta}$.\ This set has an interesting
decomposition into two curvilinear triangles when $\beta=0,$ or into a
curvilinear triangle and curvilinear bigon for $\left\vert \beta\right\vert
\in(0,\pi/2]$. Moreover, the triangle is nowhere convex for $\left\vert
\beta\right\vert \in(0,\pi/2],$ and when $\beta=\pi/2$ the bigon has
reflectional symmetry.

\setcounter{figure}{7}

\begin{definition}
\label{fundamental}For an interval $I$ with $\left\vert I\right\vert <2\pi,$
define the sector $S_{I}$ of the closed unit disk $\bar{U}$ to be
$S_{I}=\left\{  z\in\bar{U}:\arg z\in I,\text{ }0\leq r\leq1\right\}  $. For
$n\in%
\mathbb{N}
,$ $n\geq3,$ $\beta\in(-\pi/2,\pi/2],$ let $f_{\beta}$ be the rosette mapping
defined in Definition \ref{defnharmonic}, and define the\textbf{\ fundamental
set of the rosette mapping }$f_{\beta}$ to be
\[
\mathcal{A}_{\beta,n}=f_{\beta}\left(  S_{[0,2\pi/n)}\right)  =\left\{
f_{\beta}\left(  z\right)  :0\leq\left\vert z\right\vert \leq1,\text{ }%
0\leq\arg z<2\pi/n\right\}  .
\]%
\begin{figure}[h]%
\centering
\includegraphics[
height=2.1793in,
width=4.3517in
]%
{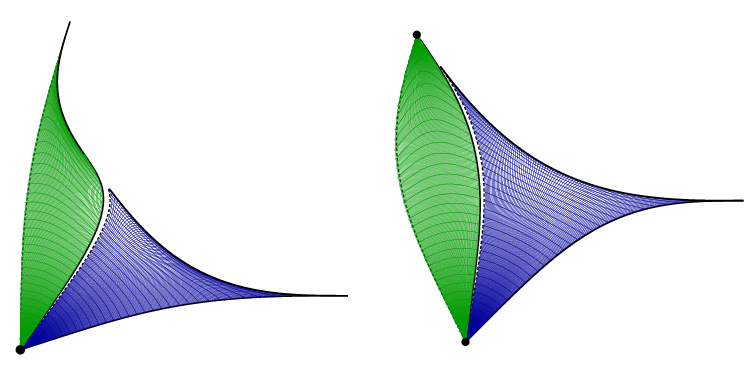}%
\caption{Fundamental sets $\mathcal{A}_{\pi/5,5}$ (left) and $\mathcal{A}%
_{\pi/2,5}$ (right)}%
\label{fund}%
\end{figure}

\end{definition}

The univalence of $f_{\beta}$ implies that the set $f_{\beta}\left(
S_{[a,b)}\right)  $ is bounded by the image of $f_{\beta}\left(  \partial
S_{[a,b)}\right)  ,$ for $\left\vert b-a\right\vert <2\pi.$ Thus the
fundamental set $\mathcal{A}_{\beta,n}$ is bounded by $f_{\beta}\left(
\partial S_{[0,2\pi/n)}\right)  .$ Definition \ref{fundamental} and the
following proposition apply only to $\beta\in(-\pi/2,\pi/2],$ but we will use
the fundamental sets $\mathcal{A}_{\beta,n}$ with $\beta$ from this restricted
range to express the graph of the rosette mapping $f_{\tilde{\beta}}\left(
\bar{U}\right)  ,$ for any $\tilde{\beta}\in%
\mathbb{R}
.$

For $\beta\in(-\pi/2,\pi/2]$ we note a further interesting decomposition in
Proposition \ref{fundunion} of the set $\mathcal{A}_{\beta,n}$ for $\beta
\neq0$ into a curvilinear triangle with sides that are nowhere convex, and a
curvilinear bigon, the latter having reflectional symmetry when $\beta=\pi/2$
(see Figure \ref{fund}).

\begin{proposition}
\label{fundunion} Let $n\in%
\mathbb{N}
,$ $n\geq3$ and $\beta\in(-\pi/2,\pi/2].$ Let $\mathcal{A}_{\beta,n}$ be the
fundamental set defined in \ref{fundamental} for the rosette mapping
$f_{\beta}$ of Definition \ref{defnhg}. The set $\mathcal{A}_{\beta,n}$ is a
curvilinear triangle subtending an angle of $2\pi/n$ at the origin. Moreover,
$\mathcal{A}_{\beta,n}$ is the disjoint union%
\[
\mathcal{A}_{\beta,n}=f_{\beta}\left(  S_{[0,\pi/n)}\right)  \cup f_{\beta
}\left(  S_{[\pi/n,2\pi/n)}\right)  .
\]
(i) For $\beta\in\left(  0,\pi/2\right)  $, the set $f_{\beta}\left(  \partial
S_{[0,\pi/n)}\right)  $ is a curvilinear triangle with sides that are nowhere
convex$.$ The set $f_{\beta}\left(  \partial S_{[\pi/n,2\pi/n)}\right)  $ is a
curvilinear bigon, where one side is strictly convex, and the other side has a
single inflection point. The angle subtended at the origin, both by the bigon
and the triangle, is $\pi/n.$ The remaining angles in the bigon and triangle
are $0.$ \newline(ii)\ If $\beta=\pi/2$ then the statements in (i) hold,
except the angle subtended by the bigon at the node of $f_{\pi/2}$ is
$\pi/2-\pi/n.$ Additionally, the bigon $f_{\pi/2}\left(  \partial
S_{[\pi/n,2\pi/n)}\right)  $ is symmetric about the line through the vertices
of the bigon$. $\newline(iii) When $\beta\in\left(  -\pi/2,0\right)  $, the
conclusions of (i) except $f_{\beta}\left(  \partial S_{[\pi/n,2\pi
/n)}\right)  $ is the curvilinear triangle and $f_{\beta}\left(  \partial
S_{[0,\pi/n)}\right)  $ is the curvilinear bigon.\newline(iv)\ When $\beta=0,$
the sides of $f_{\beta}\left(  \partial S_{[0,\pi/n)}\right)  $ and $f_{\beta
}\left(  \partial S_{[\pi/n,2\pi/n)}\right)  $ incident with the origin are
line segments, and the curvilinear triangles $f_{\beta}\left(  \partial
S_{[0,\pi/n)}\right)  $ and $f_{\beta}\left(  \partial S_{[\pi/n,2\pi
/n)}\right)  $ are reflections of one another in the line containing their
common side.
\end{proposition}

\begin{proof}
We let $\alpha_{\beta}\left(  t\right)  =f_{\beta}\left(  e^{it}\right)  $ on
$\partial U.$ The proofs that follow combine facts about the image of
$f_{\beta}$ along radial lines of $U$ in Lemma \ref{convexconcave}, and limits
of $\arg\alpha_{\beta}^{\prime}$ at cusps and nodes. The image $\mathcal{A}%
_{\beta,n}=f_{\beta}\left(  S_{[0,2\pi/n)}\right)  $ is bounded by $f_{\beta
}\left(  \partial S_{[0,2\pi/n)}\right)  .$ Since $f_{\beta}\left(  1\right)
$ and $f_{\beta}\left(  e^{i2\pi/n}\right)  $ are distinct, $f_{\beta}\left(
\partial S_{[0,2\pi/n)}\right)  $ is a curvilinear triangle (even when
$\beta=\pi/2$). At the origin, Lemma \ref{convexconcave} (ii) states that
$f_{\beta}\left(  r\right)  $ and $f_{\beta}\left(  re^{i\pi/n}\right)  $ have
tangents with arguments $\beta/2$ and $\beta/2+\pi/n.$ The side $f_{\beta
}\left(  re^{i2\pi/n}\right)  $ is a rotation by $2\pi/n$ of $f_{\beta}\left(
r\right)  ,$ and so the argument of its tangent at $0$ is $\beta/2+2\pi/n.$
Thus the angle subtended by the vertex of $\mathcal{A}_{\beta,n}$ at the
origin is $2\pi/n$. We now prove (i). Our observations about $\frac{\partial
}{\partial r}\arg f_{\beta}\left(  r\right)  ,$ $\frac{\partial}{\partial
r}\arg f_{\beta}\left(  re^{i\pi/n}\right)  ,$ and $\frac{\partial}{\partial
r}\arg f_{\beta}\left(  re^{i2\pi/n}\right)  $ at the origin show that the
angle subtended by $f_{\beta}\left(  \partial S_{[0,\pi/n)}\right)  $ and by
$f_{\beta}\left(  \partial S_{[\pi/n,2\pi/n)}\right)  $ at the origin is
$\pi/n.$ At the boundary, we have by equation (\ref{nodedir}) with $k=1$ in
Corollary \ref{cuspdir} that $\arg\alpha_{\beta}^{\prime}\left(  t\right)
\rightarrow\pi/2+\pi/n$ as $t\rightarrow\pi/n^{+}.$ This matches
$\frac{\partial}{\partial r}\arg f_{\beta}\left(  re^{i\pi/n}\right)  $ at
$f_{\beta}\left(  e^{i\pi/n}\right)  =\alpha_{\beta}\left(  \pi/n\right)  $,
by Lemma \ref{convexconcave} (ii)$.$ Thus $f_{\beta}\left(  re^{i\pi
/n}\right)  $ and $\alpha_{\beta}\left(  t\right)  $ join to form a single
smooth curve, together forming one side of $f_{\beta}\left(  \partial
S_{[0,2\pi/n)}\right)  ,$ whence the bigon. From Lemma \ref{convexconcave}
(ii), $\frac{\partial}{\partial r}\arg f_{\beta}\left(  re^{i\pi/n}\right)  $
is strictly increasing, yet $\arg\alpha_{\beta}^{\prime}\left(  e^{it}\right)
$ is strictly decreasing by (\ref{compass}). Thus $f_{\beta}\left(  e^{i\pi
/n}\right)  $ is an inflection point where curvature changes sign on the
bigon. We also have $\frac{\partial}{\partial r}\arg f_{\beta}\left(
r\right)  $ is strictly decreasing, while $\frac{\partial}{\partial r}\arg
f_{\beta}\left(  re^{i\pi/n}\right)  $ is strictly increasing. Thus the sides
of the triangle $f_{\beta}\left(  \partial S_{[0,\pi/n)}\right)  $ are nowhere
convex. Returning to the bigon, we now see that its second side $f_{\beta
}\left(  re^{i2\pi/n}\right)  $ is strictly convex, since $f_{\beta}\left(
re^{i2\pi/n}\right)  $ is a rotation about the origin of $f_{\beta}\left(
r\right)  $. We complete the proof of (i) by observing that when $\beta
\in\lbrack0,\pi/2),$ the vertices $f_{\beta}\left(  e^{i2\pi/n}\right)  $ and
$f_{\beta}\left(  1\right)  $ are cusps of $\alpha_{\beta}, $ and the angle
subtended at $f_{\beta}\left(  e^{i2\pi/n}\right)  $ and $f_{\beta}\left(
1\right)  $ is $0.$

Now we turn to (ii), with $\beta=\pi/2.$ We observed reflectional symmetry in
Theorem \ref{symmetry} (iii) of $f_{\pi/2}\left(  re^{-i\pi/n}\right)  $ and
$f_{\pi/2}\left(  r\right)  $ about the axis through $0$ and the node
$f_{\pi/2}\left(  1\right)  .$ Rotating by $2\pi/n$, the ray through $0$ and
the node $f_{\pi/2}\left(  e^{i2\pi/n}\right)  $ is also an axis of
reflectional symmetry, in which the sides $f\left(  re^{i\pi/n}\right)  $ and
$f\left(  re^{i2\pi/n}\right)  $ of the bigon are reflected. To see the
subtended angle at the boundary $\alpha_{\pi/2}$, we use Lemma
\ref{convexconcave} (ii) as above to see the tangent of $f_{\pi/2}\left(
re^{i\pi/n}\right)  $ approaches $\pi/2+\pi/n$ as $r\rightarrow1^{-}.$ By the
formula in the proof of Corollary \ref{nodes} with $k=1$, $\arg\tilde{\alpha
}_{\pi/2}^{\prime}\left(  t\right)  $ also approaches $\pi/2+\pi/n$ as
$t\rightarrow2\pi/n^{-}.$ We conclude that the side $f_{\pi/2}\left(
re^{i\pi/n}\right)  $ of the bigon becomes tangent to the boundary
$\tilde{\alpha}_{\pi/2}$ at the node $f_{\pi/2}\left(  e^{i\pi/n}\right)
=f_{\pi/2}\left(  e^{i2\pi/n}\right)  $ as $t\rightarrow2\pi/n^{-}.$ By
reflectional symmetry, the second side of the bigon $f_{\pi/2}\left(
re^{i2\pi/n}\right)  $ becomes tangent to the boundary $\tilde{\alpha}_{\pi
/2}$ at as $t\rightarrow2\pi/n^{+}.$ Thus the angle subtended in the bigon at
$f_{\pi/2}\left(  e^{i2\pi/n}\right)  $ is the same as the interior angle of
the nodes of $\tilde{\alpha}_{\pi/2},$ namely $\pi/2-\pi/n$ by Corollary
\ref{nodes}$.$

The statements in (iii) follow readily using reflectional symmetry. For
(iv)\ we already noted in Lemma \ref{convexconcave} (iii) that the images of
$f_{0}\left(  r\right)  ,$ $f_{0}\left(  re^{i\pi/n}\right)  $ and
$f_{0}\left(  re^{i2\pi/n}\right)  $ are linear. Reflectional symmetry was
established in Corollary \ref{noreflect}.
\end{proof}

We finish by decomposing the graph of an arbitrary rosette mapping$\ $into a
disjoint union of $n$ rotations of a fundamental set of Definition
\ref{fundamental}.

\begin{corollary}
Let $n\geq3$ and $\tilde{\beta}\in%
\mathbb{R}
.$ The set $f_{\tilde{\beta}}\left(  \bar{U}\right)  $ can be written as a
disjoint union of rotations of $\mathcal{A}_{\beta,n}.$ Specifically if
$\tilde{\beta}=\beta+l\pi$ for $\beta\in(-\pi/2,\pi/2]$ then
\[
f_{\tilde{\beta}}\left(  \bar{U}\right)  =i^{l}%
{\displaystyle\bigcup_{k=1}^{n}}
e^{i\left(  2k+l\right)  \pi/n}\mathcal{A}_{\beta,n}.
\]

\end{corollary}

\begin{proof}
By Corollary \ref{transit} we have $f_{\tilde{\beta}}\left(  z\right)
=e^{il\left(  \pi/n+\pi/2\right)  }f_{\beta}\left(  e^{-il\pi/n}z\right)  .$
Thus
\[
f_{\tilde{\beta}}\left(  \bar{U}\right)  =f_{\tilde{\beta}}\left(  e^{il}%
\bar{U}\right)  =e^{il\left(  \pi/n+\pi/2\right)  }f_{\beta}\left(  \bar
{U}\right)  =i^{l}e^{il\pi/n}%
{\displaystyle\bigcup_{k=1}^{n}}
e^{i2k\pi/n}\mathcal{A}_{\beta,n}.
\]

\end{proof}

\newpage

\bibliographystyle{amsalpha}
\bibliography{acompat,JaneHarmonicRefs,JaneHyperGeo,JaneMinSurfApps}

\newcommand{\etalchar}[1]{$^{#1}$}
\newif\ifabfull\abfulltrue
\providecommand{\bysame}{\leavevmode\hbox to3em{\hrulefill}\thinspace}
\providecommand{\MR}{\relax\ifhmode\unskip\space\fi MR }
\providecommand{\MRhref}[2]{%
  \href{http://www.ams.org/mathscinet-getitem?mr=#1}{#2}
}
\providecommand{\href}[2]{#2}
\begin{thebibliography}{BDM{\etalchar{+}}12}

\bibitem[AM21]{AbdullahMcDougall}
Sohair Abdullah and Jane McDougall, \emph{Rosette minimal surfaces}, 2021, In
  Preparation.

\bibitem[BDM{\etalchar{+}}12]{ECA}
Michael~A. Brilleslyper, Michael~J. Dorff, Jane~M. McDougall, James~S. Rolf,
  Lisbeth~E. Schaubroeck, Richard~L. Stankewitz, and Kenneth Stephenson,
  \emph{Explorations in complex analysis}, Classroom Resource Materials Series,
  Mathematical Association of America, Washington, DC, 2012. \MR{2963949}

\bibitem[BLW15]{MappingFull}
Daoud Bshouty, Erik Lundberg, and Allen Weitsman, \emph{A solution to
  {S}heil-{S}mall's harmonic mapping problem for polygons}, Proc. Amer. Math.
  Soc. \textbf{143} (2015), no.~12, 5219--5225. \MR{3411139}

\bibitem[CSS84]{ClunieSheilSmall}
J.~Clunie and T.~Sheil-Small, \emph{Harmonic univalent functions}, Ann. Acad.
  Sci. Fenn. Ser. A I Math. \textbf{9} (1984), 3--25. \MR{85i:30014}

\bibitem[DHL96]{DurenHenLaugesen}
Peter Duren, Walter Hengartner, and Richard~S. Laugesen, \emph{The argument
  principle for harmonic functions}, Amer. Math. Monthly \textbf{103} (1996),
  no.~5, 411--415. \MR{97f:30002}

\bibitem[DMS05]{JaneBethPeter}
Peter Duren, Jane McDougall, and Lisbeth Schaubroeck, \emph{Harmonic mappings
  onto stars}, J. Math. Anal. Appl. \textbf{307} (2005), no.~1, 312--320.
  \MR{2138992 (2006c:31002)}

\bibitem[DT00]{DurenThygerson}
Peter Duren and William~R. Thygerson, \emph{Harmonic mappings related to
  {S}cherk's saddle-tower minimal surfaces}, Rocky Mountain J. Math.
  \textbf{30} (2000), no.~2, 555--564. \MR{MR1786997 (2001i:58019)}

\bibitem[Dur04]{PeterHMBook}
Peter Duren, \emph{Harmonic mappings in the plane}, Cambridge Tracts in
  Mathematics, vol. 156, Cambridge University Press, Cambridge, 2004. \MR{2 048
  384}

\bibitem[HS86]{HenSchob}
W.~Hengartner and G.~Schober, \emph{On the boundary behavior of
  orientation-preserving harmonic mappings}, Complex Variables Theory Appl.
  \textbf{5} (1986), no.~2-4, 197--208. \MR{846488}

\bibitem[JS66]{JS}
Howard Jenkins and James Serrin, \emph{Variational problems of minimal surface
  type. {II}. {B}oundary value problems for the minimal surface equation},
  Arch. Rational Mech. Anal. \textbf{21} (1966), 321--342. \MR{MR0190811 (32
  \#8221)}

\bibitem[Lew36]{Lewy}
H.~Lewy, \emph{On the non-vanishing of the jacobian in certain one-to-one
  mappings}, Bull. Amer. Math. Soc. \textbf{42} (1936), no.~1, 689--692.

\bibitem[McD12]{Mapping4}
Jane McDougall, \emph{Harmonic mappings with quadrilateral image}, Complex
  analysis and potential theory, CRM Proc. Lecture Notes, vol.~55, Amer. Math.
  Soc., Providence, RI, 2012, pp.~99--115. \MR{2986895}

\bibitem[MS61]{MerkesScott}
E.~P. Merkes and W.~T. Scott, \emph{Starlike hypergeometric functions}, Proc.
  Amer. Math. Soc. \textbf{12} (1961), 885--888. \MR{0143950}

\bibitem[MS08]{MSStars}
Jane McDougall and Lisbeth Schaubroeck, \emph{Minimal surfaces over stars}, J.
  Math. Anal. Appl. \textbf{340} (2008), no.~1, 721--738. \MR{2376192}

\bibitem[Rai71]{Rainville}
Earl~D. Rainville, \emph{Special functions}, first ed., Chelsea Publishing Co.,
  Bronx, N.Y., 1971. \MR{0393590}

\bibitem[Sch72]{Schwarz}
H.~A. Schwarz, \emph{Gesammelte mathematische {A}bhandlungen. {B}and {I},
  {II}}, Chelsea Publishing Co., Bronx, N.Y., 1972, Nachdruck in einem Band der
  Auflage von 1890. \MR{MR0392470 (52 \#13287)}

\bibitem[SS89]{Sheil-Small}
T.~Sheil-Small, \emph{On the {F}ourier series of a step function}, Michigan
  Math. J. \textbf{36} (1989), no.~3, 459--475. \MR{91b:30002}

\end{thebibliography}

\end{document}